\documentclass[table]{amsart}

\makeatletter
\newcommand{\mylabel}[2]{#2\def\@currentlabel{#2}\label{#1}}
\makeatother
\usepackage{tikz-cd}
\usepackage{amssymb,stmaryrd,mathrsfs}
\usepackage{comment}
\usepackage{float}
\usepackage{longtable} 
\usepackage{pdflscape} 
\usepackage{orcidlink}
\usepackage{booktabs}
\usepackage{hyperref}
\definecolor{vegasgold}{rgb}{0.77, 0.7, 0.35}
\definecolor{darkgoldenrod}{rgb}{0.72, 0.53, 0.04}
\definecolor{gold(metallic)}{rgb}{0.83, 0.69, 0.22}
\hypersetup{
 colorlinks=true,
 linkcolor=darkgoldenrod,
 filecolor=brown,      
 urlcolor=gold(metallic),
 citecolor=darkgoldenrod,
 }
\newtheorem{lthm}{Theorem}

\usepackage[all,cmtip]{xy}

\DeclareFontFamily{U}{wncy}{}
\DeclareFontShape{U}{wncy}{m}{n}{<->wncyr10}{}
\DeclareSymbolFont{mcy}{U}{wncy}{m}{n}
\DeclareMathSymbol{\Sh}{\mathord}{mcy}{"58}
\usepackage[T2A,T1]{fontenc}
\usepackage[OT2,T1]{fontenc}

\newtheorem{theorem}{Theorem}[section]
\newtheorem{lemma}[theorem]{Lemma}
\newtheorem{ass}[theorem]{Assumption}
\newtheorem{definition}[theorem]{Definition}
\newtheorem{corollary}[theorem]{Corollary}
\newtheorem{remark}[theorem]{Remark}

\newtheorem{conjecture}[theorem]{Conjecture}
\newtheorem{proposition}[theorem]{Proposition}

\newcommand{\Rfine}[2]{\mathcal{R}_{p^\infty}(#1/#2)}

\newcommand{\cK}{\mathcal{K}}

\newcommand{\cH}{\mathcal{H}}

\newcommand{\Z}{\mathbb{Z}}

\newcommand{\Q}{\mathbb{Q}}
\newcommand{\F}{\mathbb{F}}

\newcommand{\cO}{\mathcal{O}}

\newcommand{\op}[1]{\operatorname{#1}}

\newcommand\mtx[4] { \left( {\begin{array}{cc}
 #1 & #2 \\
 #3 & #4 \\
 \end{array} } \right)}

\numberwithin{equation}{section}

\begin{document}

\title[Massey products and Iwasawa theory]{Massey products and the Iwasawa theory of fine Selmer groups}

\author[A.~Ray]{Anwesh Ray\, \orcidlink{0000-0001-6946-1559}}
\address[Ray]{Chennai Mathematical Institute, H1, SIPCOT IT Park, Kelambakkam, Siruseri, Tamil Nadu 603103, India}
\email{anwesh@cmi.ac.in}

\author[R.~Sujatha]{R.~Sujatha\, \orcidlink{0000-0003-1221-0710}}
\address[Sujatha]{Department of Mathematics \\ University of British Columbia \\
  Vancouver BC, V6T 1Z2, Canada.} 
  \email{sujatha@math.ubc.ca} 


\maketitle

\begin{abstract}
 An important conjecture in Iwasawa theory predicts that the fine Selmer group of any $p$-adic Galois representation is cotorsion over the relevant Iwasawa algebra with vanishing $\mu$-invariant, generalizing Iwasawa’s original conjecture for class groups. In this article, we recast this conjecture in terms of higher Galois cohomological operations called Massey products, stated purely in terms of the residual representation. This characterization implies that if the $\mu$-invariant vanishes for a given $\mathbb{Z}_p$-extension, then it also vanishes for all $\mathbb{Z}_p$-extensions the Greenberg neighbourhood of radius $1/p$. Furthermore, we establish that the unobstructedness of deformation rings over $\mathbb{Z}_p$-extensions is equivalent to the vanishing of the $\mu$-invariant of the fine Selmer group attached to the adjoint representation. For ordinary deformation rings attached to $2$-dimensional automorphic Galois representations, we establish a close connection between the Noetherian property and the vanishing of the $\mu$-invariant.
\end{abstract}

\section{Introduction}
\par Given a number field $K$ and a prime $p$, the pioneering work of Iwasawa \cite{IwasawaAnnals} initiated a systematic study of the growth of $p$-primary class groups in infinite towers of number fields. If $K_\infty/K$ is a $\Z_p$-extension with finite layers $K_n$, then the inverse limit of the $p$-primary parts of the class groups, $X(K_\infty):=\varprojlim_n \op{Cl}(K_n)[p^\infty]$, is naturally a module over the Iwasawa algebra which is a power series ring in one variable over $\Z_p$. Via this algebraic perspective, Iwasawa proved that for sufficiently large values of $n$, the growth of the $p$-part of the class group is governed by the formula
\[\#\op{Cl}(K_n)[p^\infty] = p^{p^n\mu + n\lambda + \nu},\]
\noindent where $\mu, \lambda \geq 0$ and $\nu \in \Z$ are invariants attached to the extension $K_\infty/K$. Among these, the $\mu$-invariant measures the depth of $p$-power divisibility in the characteristic power series of $X(K_\infty)$. Iwasawa conjectured that in the cyclotomic $\Z_p$-extension of any number field, $\mu=0$. This conjecture is established by Ferrero and Washington \cite{ferrerowashington} when $K/\Q$ is abelian, however is wide open in general. 

Over the decades, Iwasawa theory has been studied in a broader framework. In this context, Selmer groups encode arithmetic information attached to elliptic curves, modular forms, and more general motives which give rise to Galois representations. Their structure as modules over Iwasawa algebras is conjecturally governed by analytic invariants of $p$-adic $L$-functions. Within this framework, the fine Selmer group occupies a particularly subtle yet important position. By construction, it is obtained from the usual $p$-primary Selmer group by imposing the strictest possible local conditions at the primes above $p$, thereby encoding only the most delicate global arithmetic information. Coates and Sujatha conjectured \cite[Conjecture A]{CoatesSujatha} that the fine Selmer group is always a cotorsion module over the relevant Iwasawa algebra, and moreover that its $\mu$-invariant vanishes. This may be viewed as a far-reaching generalization of Iwasawa’s original $\mu=0$ conjecture for class groups.

\par In recent years, there has been an increase in interest in the relationship between Iwasawa theory and explicit operations in Galois cohomology. Classical cup products already play a significant role in controlling Selmer groups, but their higher-order generalizations called Massey products have emerged as a powerful tool for uncovering a more refined arithmetic structure. The breakthrough work of McCallum and Sharifi \cite{McCallumSharifi} demonstrated explicit connections between cup products in Galois cohomology and Iwasawa $\lambda$-invariants of cyclotomic fields. Massey products offer a natural language for encoding deeper structural information in Iwasawa theory (see, for example, \cite{SharifiMassey, qi2025iwasawa}). They also naturally arise in various other related contexts. For instance, Massey products are known to capture information concerning the rank of the Eisenstein ideal \cite{WWE}. The study of higher order Massey products has revealed deep connections with the structure of absolute Galois groups and class groups of number fields (cf. for instance, work of Min\'ac and T\^an \cite{MinacTan1, MinacTan2, MinacTan4, MinacTan3}).

The present article develops these themes by integrating the study of Iwasawa cohomology with the generalized Bockstein maps introduced in the work of Lam, Liu, Sharifi, Wake, and Wang \cite{Lametal}. Our main result shows that Conjecture A in \cite{CoatesSujatha} can be reformulated in terms of Massey products. For details and precise notation, see subsection \ref{section 6.2 Massey} and Theorem \ref{main thm of paper}.

\begin{lthm}[Theorem \ref{main thm of paper}]
Let $K$ be a number field and $K_\infty/K$ a $\Z_p$-extension, and let $S$ be a finite set of primes containing those above $p$. Assume that all primes $v \mid p$ of $K$ are totally ramified in $K_\infty$, and that all primes in $S$ are finitely decomposed in $K_\infty$. Let $\cK/\Q_p$ be a finite extension with valuation ring $\cO$, and let $\rho:\op{Gal}(\bar{K}/K)\to \op{GL}_d(\cO)$ be a continuous Galois representation which is unramified outside $S$. Let $\bar{\rho}$ be the residual representation of $\rho$, and $\bar{\rho}^*$ denote the Tate-dual of $\bar{\rho}$. Let \[\kappa:\op{Gal}(\bar{K}/K)\twoheadrightarrow \op{Gal}(K_\infty/K)\xrightarrow{\sim} \Z_p\]
\noindent be the canonical surjection. Then the following are equivalent:
\begin{enumerate}
\item The fine Selmer group of $\rho$ over $K_\infty$ is $\Lambda$-cotorsion with vanishing $\mu$-invariant.
\item For some $n>0$, the Massey products $(\underbrace{\kappa, \kappa, \dots, \kappa}_{n-\text{times}}, \lambda)_\tau$ span $H^2(K_S/K,\bar{\rho}^*)$, where $\tau$ ranges over certain defining systems and
$\lambda$ ranges over $H^1(K_S/K, \bar{\rho}^*)
    $.
\end{enumerate}
\end{lthm}
\par Given a number field $K$, one considers the set of all $\Z_p$-extensions of $K$. Greenberg endowed this set with a topology. We now mention a striking consequence of the above result. Suppose that condition (1) is satisfied for $\rho$ over a $\Z_p$-extension $K_\infty/K$. We then show that the same is true for a uniform large Greenberg neighbourhood of $K_\infty/K$. Indeed, if there are two $\Z_p$-extensions $K_\infty^1$ and $K_\infty^2$ which agree at the first layer, then (1) holds for $\rho$ over $K_\infty^1$ if and only if (1) holds for $\rho$ over $K_\infty^2$.
\begin{lthm}\label{thmb}[Theorem \ref{greenberg neighbourhood theorem}]
  Let $K$ be a number field, $S$ a finite set of primes of $K$ containing those above $p$, and $\rho:\op{G}_{K,S}\to \op{GL}_d(\cO)$ a continuous Galois representation. Let $K_\infty^1, K_\infty^2$ be $\Z_p$-extensions of $K$ such that:
  \begin{enumerate}
\item all $v|p$ are totally ramified and all $v\in S$ are finitely decomposed in in $K_\infty^i$, for $i=1,2$,
\item $K_1^1=K_1^2$.
\end{enumerate}
If the fine Selmer group for $\rho$ over $K_\infty^1$ is cotorsion with $\mu=0$, then the same holds for the fine Selmer group for $\rho$ over $K_\infty^2$. In particular, the vanishing of the $\mu$-invariant is a property which is locally constant in the Greenberg topology in uniformly large Greenberg neighbourhoods.
\end{lthm}

\par We also explore the broader consequences of Conjecture A of Coates and Sujatha \cite{CoatesSujatha} for the deformation theory of Galois representations considered over a $\Z_p$-extension. We illustrate a bridge between two a priori distinct frameworks: the structure theory of Iwasawa modules and the geometry of deformation spaces. Specifically, Proposition \ref{unobstructedness for def rings} establishes that the unobstructedness of certain deformation rings over $K_\infty$ is equivalent to the vanishing of the $\mu$-invariant of the fine Selmer group attached to the Tate-twist of the adjoint representation. 
The study of deformation theory in infinite towers of number fields has its origins in the work of Hida \cite{Hida}, who investigated deformations of $p$-adic Galois representations over the cyclotomic $\Z_p$-extension. Burungale--Clozel \cite{BCM1} and Burungale--Clozel--Mazur \cite{BCM2}, further analyzed ordinary deformations over $p$-adic Lie extensions. To recall their setting, let $\F_q$ be a finite field with ring of Witt vectors $\op{W}(\F_q)$, let $F$ be a totally real number field, and let $F_{\op{cyc}}$ denote its cyclotomic $\Z_p$-extension. Suppose that $\bar{r}:\op{Gal}(\bar{F}/F)\to \op{GL}_2(\F_q)$ is a $p$-ordinary residual representation satisfying certain natural hypotheses. Then Burungale and Clozel show that the ordinary deformation problem for $\bar{r}|_{\op{Gal}(\bar{F}/F_{\op{cyc}})}$ is pro-representable by a universal deformation ring $R_\infty$, which arises as the inverse limit $\varprojlim_n R_n$ of the ordinary deformation rings $R_n$ associated to the finite layers $F_n$ of $F_{\op{cyc}}/F$. Each $R_n$ is a Noetherian $\op{W}(\F_q)$-algebra, but this does not automatically guarantee that $R_\infty$ itself is Noetherian. The main theorem of \cite{BCM1} shows that, provided $R_\infty$ is Noetherian and certain additional conditions are satisfied, one can in fact identify $R_\infty$ with a power series ring of the form $
\op{W}(\F_q)\llbracket T_1,\dots,T_s\rrbracket
$. Burungale--Clozel--Mazur later demonstrate that $s$ can in fact be arbitrarily large. In Proposition \ref{mainpropsection5}, we prove that in this setting the universal deformation ring $R_\infty$ is Noetherian if and only if the $\mu$-invariant of the Greenberg Selmer group of the adjoint representation of $\rho$ vanishes. Moreover, Corollary \ref{cor5.8} asserts that $R_\infty$ is a formal power series ring in a finite number of variables if and only if the $\mu$-invariant of the adjoint representation is zero.
\par Beyond the immediate implications for the geometry of deformation rings, this perspective also suggests the possibility of more refined results, since Massey products in Galois cohomology provide a powerful framework for probing the structure of Galois groups. Given that deformation problems are governed by obstruction classes, it is natural to expect that Massey products should play a significant role in their description. We believe that the connections uncovered in this work will serve as a starting point for future developments along these lines.

\subsection*{Data availability} The article has no associated data. 

\subsection*{Conflict of interest statement} We have no conflict of interest to report.

\subsection*{Acknowledgement} The second named author gratefully acknowledges support from NSERC Discovery grant 2019-03987. The authors thank the referee for the careful reading of the manuscript, the helpful comments, and the valuable suggestions for improvement.

\section{Cup products and Iwasawa theory of class groups}
\par Let $K$ be a number field. Fix an algebraic closure $\overline{K}$ of $K$, and set $\operatorname{G}_K := \operatorname{Gal}(\overline{K}/K)$. Let $\Omega_K$ denote the set of finite primes of $K$. For each $v \in \Omega_K$, fix an algebraic closure $\overline{K_v}$ of $K_v$, and write $\operatorname{G}_{K_v} := \operatorname{Gal}(\overline{K_v}/K_v)$. For $n\geq 1$, let $\mu_{p^n}\subset \overline{K}$ denote the $p^n$-th roots of unity and set $\mu_{p^\infty}:=\bigcup_{n\geq 1} \mu_{p^n}$. The field $K(\mu_{p^\infty})=\bigcup_{n\geq 1} K(\mu_{p^n})$ is a $\Z_p$-extension of $K(\mu_p)$. The cyclotomic $\Z_p$-extension of $K$ is the unique $\Z_p$-extension $K_{\op{cyc}}/K$ which is contained in $K(\mu_{p^\infty})$.

\par Let $K_\infty/K$ be a $\Z_p$-extension. For $n\geq 0$, let $K_n\subset K_\infty$ be the subfield such that $\op{Gal}(K_n/K)\simeq \Z/p^n \Z$. Thus we have a tower of Galois extensions:
\[K=K_0\subset K_1\subset \dots \subset K_n \subset K_{n+1}\subset \dots \subset K_\infty=\bigcup_{n\geq 0} K_n,\] and refer to $K_n$ as the \emph{$n$-th layer}.
Set $\Gamma:=\op{Gal}(K_\infty/K)$ and identify $\op{Gal}(K_n/K)$ with $\Gamma/\Gamma^{p^n}$. The Iwasawa algebra $\Lambda$ is defined to be inverse limit of finite group rings: \[\Lambda:=\Z_p\llbracket \Gamma \rrbracket =\varprojlim_n \Z_p[\Gamma/\Gamma^{p^n}]=\varprojlim_n \Z_p[\op{Gal}(K_n/K)].\] Let $\gamma$ be a topological generator of $\Gamma$ and set $T:=\gamma-1\in \Lambda$. Then, $\Lambda$ is identified with the formal power series ring $\Z_p\llbracket T\rrbracket$. Given a finitely generated $\Lambda$-module $\mathfrak{M}$, there exists a $\Lambda$-module homomorphism

$$
\mathfrak{M} \longrightarrow \Lambda^\alpha \oplus \left( \bigoplus_{i=1}^s \Lambda / (p^{\mu_i}) \right) \oplus \left( \bigoplus_{j=1}^t \Lambda / (f_j(T)) \right)
$$
\noindent with finite kernel and cokernel (see \cite[Ch.~13]{introcycfields}). Here, $\alpha:= \operatorname{rank}_\Lambda \mathfrak{M}$, and each $f_j(T)$ is a \emph{distinguished polynomial}, that is, a monic polynomial in $\Z_p[T]\subset \Lambda$ whose non-leading coefficients are all divisible by $p$. Note that $\mathfrak{M}$ is a torsion $\Lambda$-module if and only if $\alpha=0$. The $\mu$-invariant of $\mathfrak{M}$ is defined as the sum $\sum_{i=1}^s \mu_i$, and taken to be $0$ when $s=0$. On the other hand, the $\lambda$-invariant is given by 
\[\lambda(\mathfrak{M}):=\sum_{j=1}^t \op{deg} f_j(T).\]
\noindent Suppose $\mathfrak{M}$ is a torsion $\Lambda$-module. Then it is straightforward to verify that $\mathfrak{M}$ is finitely generated as a $\Z_p$-module if and only if $\mu(\mathfrak{M}) = 0$. In particular, $\mu(\mathfrak{M}) = 0$ if and only if the quotient $\mathfrak{M}/p\mathfrak{M}$ is finite. In this case, the $\lambda$-invariant of $\mathfrak{M}$ equals the $\Z_p$-rank of $\mathfrak{M}$ and satisfies
\[\lambda(\mathfrak{M}) = \operatorname{rank}_{\Z_p} \mathfrak{M} \leq \dim_{\F_p}(\mathfrak{M}/p\mathfrak{M}),\]
\noindent with equality if and only if $\mathfrak{M}$ contains no nontrivial finite $\Lambda$-submodules.

\par Given any $\Lambda$-module $\mathfrak{M}$, let $\mathfrak{M}^\vee$ denote its Pontryagin dual $\op{Hom}_{\Z_p}(\mathfrak{M}, \Q_p/\Z_p)$. We say that $\mathfrak{M}$ is \emph{cofinitely generated} (resp. \emph{cotorsion}) if $\mathfrak{M}^\vee$ is finitely generated (resp. torsion) as a $\Lambda$-module. When the context makes this unambiguous, we shall write $\mu(\mathfrak{M})$ (resp. $\lambda(\mathfrak{M})$) to mean the $\mu$-invariant (resp. $\lambda$-invariant) of $\mathfrak{M}^\vee$. This is the convention we adopt for fine Selmer groups over $\Z_p$-extensions, which are the main objects of interest in this paper.

\par Let $S_p(K)$ denote the set of primes of $K$ lying above $p$, and let $S$ be a finite set of primes of $K$ containing $S_p(K)$. Given an algebraic extension $F/K$, denote by $S(F)$ the set of primes of $F$ lying above primes in $S$. Note that all primes of $K$ are finitely decomposed in $K_{\op{cyc}}$, and thus $S(K_{\op{cyc}})$ is a finite set of primes. On the other hand, for an arbitrary $\Z_p$-extension $K_\infty/K$, the set $S(K_\infty)$ may be infinite. For instance, this is indeed the case when $K$ is an imaginary quadratic field, $K_\infty$ is the anticyclotomic $\Z_p$-extension of $K$ and $S$ contains a prime $v\nmid p$ which is inert over $\Q$. Let $\operatorname{Cl}_S(F)$ denote the $S$-class group of $F$, and let $H_p^S(F)$ be the maximal abelian $p$-extension of $F$ in which all primes in $S(F)$ split completely. By class field theory, there is a natural isomorphism

$$
\operatorname{Cl}_S(F)[p^\infty] \simeq \operatorname{Gal}(H_p^S(F)/F).
$$
\noindent Define $H_p^S(K_\infty) := \bigcup_n H_p^S(K_n)$ and set $
\mathcal{X}^S(K_\infty) := \operatorname{Gal}(H_p^S(K_\infty)/K_\infty)$.
\noindent Note that 
\[\mathcal{X}^S(K_\infty)\simeq \varprojlim_n \op{Gal}(H_p^S(K_n)/K_n)\simeq \varprojlim_n \op{Cl}_S(K)[p^\infty],\] where the last inverse limit is taken with respect to norm maps. For ease of notation, we set $\mathcal{X}(K_\infty):=\mathcal{X}^\emptyset(K_\infty)$. Iwasawa \cite{IwasawaAnnals} proved that $\mathcal{X}(K_\infty)$ is a torsion $\Lambda$-module and that there are constants $\mu\in \Z_{\geq 0}$, $\lambda\in \Z_{\geq 0}$ and $\nu\in \Z$ such that for all sufficiently large values of $n$, 
\[\# \op{Cl}(K_n)[p^\infty]=p^{p^n \mu+n \lambda+\nu},\] where $\mu=\mu(\mathcal{X}(K_\infty))$ and $\lambda=\lambda(\mathcal{X}(K_\infty))$. The $\nu$-invariant on the other hand is an integer which is possibly negative. Moreover, for the cyclotomic $\Z_p$-extension, the $\mu$-invariant of $\mathcal{X}(K_{\op{cyc}})$ is conjectured to be zero. Ferrero and Washington \cite{ferrerowashington} proved this when $K/\Q$ is an abelian extension. 

\par We now specialize to the case where $K = \Q(\mu_p)$ is the $p$-th cyclotomic field and $K_\infty = \Q(\mu_{p^\infty}) = K_{\operatorname{cyc}}$ is its cyclotomic $\Z_p$-extension. In this setting, the Herbrand–Ribet theorem relates the nontriviality of certain eigenspaces of the $p$-part of the class group of $K$ to the divisibility of Bernoulli numbers by $p$. The subgroup $C_K \subset \Z[\mu_p, 1/p]^\times$ of cyclotomic $p$-units embeds via Kummer theory into Galois cohomology, where cup product pairings capture subtle information about class groups. Building on this, McCallum and Sharifi \cite{McCallumSharifi} introduced a pairing on cyclotomic $p$-units that relates the nontriviality of class group eigenspaces to the behavior of certain cup products in Galois cohomology. They further showed that the Iwasawa $\lambda$-invariants governing the growth of class groups in the cyclotomic tower can be characterized in terms of the vanishing of these cup products. We now describe this connection in more detail.

\par Note that $K_n=\Q(\mu_{p^{n+1}})$. Let $\Delta := \operatorname{Gal}(K/\Q) \cong (\Z/p\Z)^\times$ denote the Galois group of $K$ over $\Q$, which is cyclic of order $(p-1)$. The mod-$p$ cyclotomic character $\bar{\chi}: \Delta \xrightarrow{\sim} (\Z/p\Z)^\times$ is an isomorphism by construction, and we write $\omega: \Delta \hookrightarrow \Z_p^\times$ for its Teichm\"uller lift. For convenience, let us denote the $p$-primary part of the class group of $K$ by $A_K := \operatorname{Cl}(K)[p^\infty]$.
\par A prime $p$ is said to be \emph{irregular} if it divides the class number of $\Q(\mu_p)$, and said to be \emph{regular} otherwise. The Galois action of $\Delta$ on $A_K$ allows us to decompose $A_K$ into a direct sum of eigenspaces under the characters $\omega^i$, for $i \in \{0,1,\dots,p-2\}$:
\[A_K = \bigoplus_{i=0}^{p-2} A_K(\omega^i),\] \noindent where $A_K(\omega^i) := \{x \in A_K : \sigma x = \omega(\sigma)^i x \text{ for all } \sigma \in \Delta\}$. We refer to $A_K(\omega^i)$ as the \emph{$i$-th eigenspace} of $A_K$. These eigenspaces are customarily classified as \emph{even} or \emph{odd} depending on the parity of $i$, i.e., the even (resp. odd) eigenspaces correspond to even (resp. odd) $i$. Accordingly, one has that
$$
A_K^+ := \bigoplus_{\substack{0 \leq i \leq p-2 \\ i \text{ even}}} A_K(\omega^i), \quad \text{and} \quad A_K^- := \bigoplus_{\substack{0 \leq i \leq p-2 \\ i \text{ odd}}} A_K(\omega^i).
$$
\noindent A famous conjecture due to Vandiver asserts that $A_K^+ = 0$, i.e., the plus part vanishes for all odd primes $p$. Let $f(x) := \frac{x}{e^x - 1}$, and define the Bernoulli numbers $B_n$ via the Taylor expansion
$$f(x) = \sum_{n=0}^\infty \frac{B_n}{n!}x^n, \quad \text{so that } B_n = f^{(n)}(0).
$$
\noindent These numbers play a central role in algebraic number theory. In particular, for positive integers $n \geq 1$, the special values of the Riemann zeta function at negative integers are given by the identity
$
\zeta(1 - n) = -\frac{B_n}{n}
$. The Herbrand–Ribet theorem states that for even integers $k$ with $2 \leq k \leq p - 3$, the $(p - k)$-th eigenspace $A_K(\omega^{p-k})$ is nonzero if and only if $p$ divides $B_k$. This connection was made more precise by Mazur and Wiles, who showed that the order of this eigenspace is given by $$
\# A_K(\omega^{p-k}) = p^{v_p(B_{1, \omega^{k-1}})},
$$
\noindent where $B_{1, \omega^{k-1}} := \frac{1}{p} \sum_{a=1}^{p-1} a \omega^{k-1}(a)$ is a generalized Bernoulli number, which satisfies the congruence
\[B_{1, \omega^{k-1}}\equiv \frac{B_k}{k}\pmod{p}.\]

\par Define $E_K:=\Z[\mu_p, 1/p]^\times$, the group of units in the ring of $p$-integers of $K = \Q(\mu_p)$. Within $E_K$, we consider the subgroup $
C_K \subseteq E_K
$ of \emph{cyclotomic $p$-units}, defined to be the subgroup generated by $-1$ and elements of the form $\zeta_p^i - 1$, where $\zeta_p := \exp(2\pi i/p)$ is a fixed primitive $p$-th root of unity and $(i,p)=1$. A fundamental theorem in the theory of cyclotomic fields states that the index $[E_K : C_K]$ is finite, and its $p$-part equals $\# A_K^+$. Let $S_p=\{(1-\zeta_p)\}$ be the prime of $K$ which lies above $p$ and $K_{S_p}\subset \bar{K}$ be the maximal extension of $K$ in which all primes $\ell\neq p$ are unramified. By Kummer theory, 
\[H^1(K_{S_p}/K, \mu_p)\simeq D_K/(K^\times)^p,\]
where
\[D_K=\{\alpha\in K^\times\mid \alpha \cO_{K, S_p}=\mathfrak{a}^p, \mathfrak{a}\subset \cO_{K, S_p}\}.\]
We note that $\op{Gal}(K_{S_p}/K)$ acts trivially on $\mu_p$, nonetheless, there is a natural action of $\Delta$ on $H^1(K_{S_p}/K, \mu_p)$. Furthermore, there is a $\Delta$-equivariant cup product pairing 
\[H^1(K_{S_p}/K, \mu_p)\times H^1(K_{S_p}/K, \mu_p)\rightarrow H^2(K_{S_p}/K, \mu_p^{\otimes 2}).\] As an $\F_p[\Delta]$-module, we identify $H^2(K_{S_p}/K, \mu_p^{\otimes 2})$ with $H^2(K_{S_p}/K, \mu_p)\otimes \mu_p$. Let $A_{K}$ be the $p$-primary part of the class group of $K$. There is a natural $\Delta$-equivariant isomorphism:
\[H^2(K_{S_p}/K, \mu_p)\simeq A_{K}/pA_{K},\] cf. \cite[equation (9)]{McCallumSharifi}. Thus one has a cup-product pairing 
\[D_K/(K^\times)^p \times D_K/(K^\times)^p\rightarrow A_{K}/pA_{K} \otimes \mu_p,\] which in turn induces a cup product pairing on cyclotomic $p$-units:
\begin{equation}\langle \cdot, \cdot \rangle_p: C_K \times C_K\rightarrow A_{K}^-\otimes \mu_p.\end{equation} The following is a rephrasing of \cite[Conjecture 5.3]{McCallumSharifi}.
\begin{conjecture}[McCallum--Sharifi]
    For all primes $p$, the pairing $\langle \cdot, \cdot\rangle_p$ is nontrivial.
\end{conjecture}
\noindent It is shown in \cite{sharifieisenstein} that the conjecture is true for primes $p<1000$. 
\par The above pairing has various connections to Iwasawa theory and is related to deep conjectures like Greenberg's pseudo-nullity conjecture. Recall that the cyclotomic $\Z_p$-extension of $K$ is given by $K_{\op{cyc}}=K(\mu_{p^\infty})$ and $K_n=\Q(\mu_{p^{n+1}})$. Let $A_n:=\op{Cl}(K_n)[p^\infty]$, note that $A_0=A_K$. Then for $n\gg 0$, 
\[\# A_n(\omega^i)=p^{n\lambda_i+\nu_i}.\] Suppose that $i$ is odd and $A_K(\omega^i)\neq 0$, then Vandiver's conjecture predicts that $A_K(\omega^i)$ is either trivial or cyclic. Let $G:=\op{Gal}(K_{S_p}/K)$ and 
\[\kappa: G\twoheadrightarrow\Gamma\simeq \Z_p\] be the natural quotient map. Denote by $\bar{\kappa}:G\rightarrow \Z/p\Z$ the mod-$p$ reduction of $\kappa$. Since $A_K(\omega^i)$ is cyclic, it follows that it gives rise to a class $\alpha\in H^1(G, \mu_p)$. The cup product $\bar{\kappa}\cup \alpha\in H^2(G, \mu_p^{\otimes 2})$ gives an element in $A_K\otimes \mu_p$. 
\par After choosing a primitive $p$-th root of unity $\zeta_p$, the mod-$p$ cyclotomic character $\bar{\kappa}:G\to \mathbb{Z}/p\mathbb{Z}$ may be identified with the Kummer class of $\zeta_p$, viewed as an element of $H^1(G,\mu_p)$. Under this identification, the cup product $\bar{\kappa}\cup\alpha$ takes values in $H^2(G,\mu_p^{\otimes 2})$. Since $K$ contains $\mu_p$, the action of $G$ on $\mu_p$ is trivial. Thus, after choosing an identification $\mu_p\simeq \mathbb{Z}/p\mathbb{Z}$, this is only a matter of normalization.

\begin{theorem}[McCallum--Sharifi]\label{MSthm}
    Assume that $i$ is odd and $A_K(\omega^i)\neq 0$ is cyclic. Then, $\lambda_i\geq 1$. Further, we have that 
    \[\lambda_i\geq 2\text{ if and only if }\bar{\kappa}\cup \alpha=0. \]
\end{theorem}
\begin{proof}
    See Proposition 4.2 and ll.~12--13 on p.~282 in \cite{McCallumSharifi}.
\end{proof}

\noindent The McCallum–Sharifi conjecture and the above theorem illustrate a connection between the structure of cyclotomic class groups and Galois cohomology via the cup product pairing. The theorem shows that the vanishing of a specific cup product class precisely controls whether the $\lambda$-invariant in a given eigenspace exceeds the minimal expected bound. This interplay between  Iwasawa theory and the pairing above motivates further investigation connecting fine Selmer groups, Iwasawa modules, and Massey products, which are the main themes of this article.

\section{Fine Selmer groups and the $\mu=0$ conjecture}
\par In this section, we fix an odd prime number $p$ and introduce the main objects of study in this article: the fine Selmer groups associated to Galois representations considered over a $\Z_p$-extension $K_\infty$ of a number field $K$. For a comprehensive introduction to fine Selmer groups, the reader may refer to \cite{CoatesSujatha}. Let $S$ be a finite set of primes of $K$ containing $S_p(K)$.
\begin{ass}\label{asszpextn} We assume that the primes $v\in S_p(K)$ are totally ramified in $K_\infty$ and that all primes $v\in S$ are finitely decomposed in $K_\infty$.
\end{ass}Throughout, for any field $F$ of characteristic zero, set $\op{G}_F:=\op{Gal}(\bar{F}/F)$. We use the shorthand notation $H^i(F, \cdot) := H^i(\operatorname{G}_F, \cdot)$. Let $K_S\subset \overline{K}$ be the maximal extension of $K$ in which all primes $v \notin S$ are unramified. We shall set $\op{G}_{K,S}:=\op{Gal}(K_S/K)$ and adopt the notation \[H^i(K_S/L, \cdot):=H^i(\op{Gal}(K_S/L), \cdot)\] for any field $L$ such that $K\subseteq L \subseteq K_S$. For $v\in \Omega_K$, set 
\[K^i_v(\cdot/L):=\prod_{w|v} H^i(L_w, \cdot),\] where $w$ runs through the primes of $L$ which lie above $v$.

\begin{definition}
    Let $M$ be a $p$-primary $\op{G}_{K,S}$-module and $L$ be an extension of $K$ which is contained in $K_S$. The fine Selmer group over $L$ associated to $M$ is defined as follows
\[\Rfine{M}{L}:=\op{ker}\left\{H^1(K_S/L, M)\longrightarrow \bigoplus_{v\in S} K^1_v(M/L)\right\}.\]
\end{definition}
\noindent Thus the fine Selmer group consists of all cohomology classes with values in $M$ that are unramified at the primes outside $S$ and trivial at all primes of $L$ that lie above $S$. Apriori, $\Rfine{M}{L}$ depends on the choice of primes $S$. 
\begin{proposition}
    Let $L/K$ be an infinite algebraic extension which contains a $\Z_p$-extension $K_\infty$ of $K$ satisfying Assumption \ref{asszpextn}. Then, $\Rfine{M}{L}$ does not depend on the choice of primes $S$.
\end{proposition}
\begin{proof}
    In order to emphasize the dependence of the fine Selmer group on $S$, we set 
    \[\mathcal{R}_{p^\infty}^S(M/L):=\op{ker}\left\{H^1(K_S/L, M)\longrightarrow \bigoplus_{v\in S} K^1_v(M/L)\right\}.\]
    Let $S'$ be a finite set of primes containing $S$ such that all $v\in S'$ are finitely decomposed in $K_\infty$. We show that 
    \[\mathcal{R}_{p^\infty}^{S'}(M/L)=\mathcal{R}_{p^\infty}^S(M/L).\] Set $S'':=S'\setminus S$ and for $v\in S''$, let $w$ be a prime of $L$ which lies above $v$ and denote by $\op{G}_w$ (resp. $\op{I}_w$) the decomposition (resp. inertia) group at $w$ and $\kappa_w$ its residue field. Note that $\mathcal{R}_{p^\infty}^{S'}(M/L)$ consists of classes $f\in \mathcal{R}_{p^\infty}^{S}(M/L)$ which when restricted to each decomposition group $\op{G}_w$ is zero for $w\in S''$. Let $f\in \mathcal{R}_{p^\infty}^{S}(M/L)$. Then, $f$ is zero when restricted to $\op{I}_w$ for all $w\in S''$. Thus the restriction of $f$ to $H^1(L_w, M)$ lies in 
\[H^1_{\op{nr}}(L_w, M):=\op{ker}\left\{H^1(L_w, M)\longrightarrow H^1(\op{I}_w, M)\right\},\]
which can be identified with $H^1(\kappa_w, M^{\op{I}_w})$. Since $v$ is finitely decomposed in $K_\infty$, it follows that $\overline{\kappa_w}/\kappa_w$ is a prime to $p$-extension and hence $H^1(\kappa_w, M^{\op{I}_w})=0$. Thus we find that $f$ is indeed $0$ when restricted to $\op{G}_w$. It follows that $f\in \mathcal{R}_{p^\infty}^{S'}(M/L)$. This shows that 
\[\mathcal{R}_{p^\infty}^{S'}(M/L)=\mathcal{R}_{p^\infty}^{S}(M/L).\]
\end{proof}

\par We shall study the algebraic structure of fine Selmer groups over $\Z_p$-extensions of $K$ satisfying Assumption \ref{asszpextn}. The module $M$ considered will be a $p$-divisible module associated with an integral Galois representations. Such representations arise naturally in geometry, in the following contexts:
\begin{enumerate}
\item[(a)] Galois representations arising from the $p$-adic Tate module $\mathbf{T}_p A=\varprojlim_n A[p^n]$ of an abelian variety $A$ defined over $\bar{\Q}$.
\item[(b)] Galois actions on the \'etale cohomology groups $H^i(X, \mathbb{Z}_p)$ of algebraic varieties $X$ defined over $\overline{\Q}$, which yield natural integral Galois representations.
\item[(c)] Through the Langlands correspondence, Galois representations are attached to certain algebraic, essentially self-dual automorphic representations. In particular, Galois representations arising from Hecke eigencuspforms provide a rich source of $2$-dimensional Galois representations.
\end{enumerate}

Let $\cK$ be a finite extension of $\mathbb{Q}_p$, and let $\mathcal{O}$ denote its valuation ring. Given a continuous integral Galois representation $\rho: \operatorname{G}_{K,S} \rightarrow \operatorname{GL}_d(\mathcal{O})$, let $\mathbf{T}_\rho$ denote the underlying free $\mathcal{O}$-module of rank $n$ on which $\operatorname{G}_{K,S}$ acts via $\rho$. Set $\mathbf{V}_{\rho}:=\mathbf{T}_\rho\otimes_{\Z_p} \Q_p$ and consider the $p$-divisible discrete $\operatorname{G}_{K,S}$-module
$$
\mathbf{A}(\rho) := \mathbf{V}_\rho/\mathbf{T}_\rho=\mathbf{T}_\rho \otimes_{\mathcal{O}} (K/\mathcal{O}).
$$
Note that when $\mathbf{T}_\rho$ is the Tate-module $\mathbf{T}_pA$ of an abelian variety $A$, then $\mathbf{A}(\rho)=A[p^\infty]$. The main object of study in this article is the fine Selmer group $\Rfine{\mathbf{A}(\rho)}{K_\infty}$, we set $Y(\mathbf{A}(\rho)/K_\infty)$ to be its Pontryagin dual. We let $\mu_\rho$ be the $\mu$-invariant
\[\mu_\rho:=\mu\left(Y(\mathbf{A}(\rho)/K_\infty)\right).\] Note that the fine Selmer group depends on the choice of lattice $\mathbf{T}_\rho$ in $\mathbf{V}_{\rho}$. When $K_\infty$ is the cyclotomic $\Z_p$-extension of $K$,  \cite[Conjecture A]{CoatesSujatha} predicts that the $\mu_\rho=0$. Denote by $\rho^\vee$ the contragredient representation and let $\rho^*:=\rho^\vee(1)$ the Tate-dual of $\rho$. Let $\varpi$ be a uniformizer of $\cO$ and $\F:=\cO/(\varpi)$. Given an $\cO$-module $M$, set $M[\varpi^n]:=\op{ker}(M\xrightarrow{\varpi^n} M)$ and set $M[\varpi^\infty]:=\bigcup_n M[\varpi^n]$. Say that $M$ is $p$-primary if $M=M[\varpi^\infty]$. We set $\mathbf{V}_{\bar{\rho}}:=\mathbf{T}_\rho\otimes_{\cO} \F$ and $\mathbf{V}_{\bar{\rho}^*}:=\mathbf{T}_{\rho^*}\otimes_{\cO} \F=\mathbf{V}_{\bar{\rho}}^\vee(1)$.
\begin{conjecture}[Weak Leopoldt Conjecture for \(\rho\)]\label{WLC} We have that \[ H^2(K_S/K_\infty, \mathbf{A}(\rho)) = 0.\]
\end{conjecture}

We refer the reader to \cite[Appendix B]{perrinriou} for a compilation of cases in which the Weak Leopoldt Conjecture has been established.
\par Let $M$ be a $p$-primary module over $\op{G}_{K,S}$ and let $H^i_{\op{Iw}}(K_S/K_\infty, M)$ be the \emph{Iwasawa cohomology group}, which is defined as the inverse limit $\varprojlim_n H^i(K_S/K_n, M)$ taken with respect to corestriction maps. 

\begin{proposition}\label{H^2 equivalence propn}
    Assume that $\rho^*$ satisfies the weak Leopoldt conjecture \ref{WLC} over $K_\infty$. Then the following conditions are equivalent:
    \begin{enumerate}
    \item $H^2(K_S/K_\infty, \mathbf{V}_{\bar{\rho}^*})=0$, 
            \item $H^2(K_S/K_\infty, \mathbf{V}_{\bar{\rho}^*})$ is finite.
    \end{enumerate}
\end{proposition}
\begin{proof}
   For ease of notation, set $M:=\mathbf{A}(\rho^*)$ and identify $M[\varpi]$ with $\mathbf{V}_{\bar{\rho}^*}$. Consider the short exact sequence of Galois modules
    \[0\rightarrow M[\varpi]\rightarrow M\xrightarrow{\varpi} M\rightarrow 0. \]
    Note that we assume that $H^2(K_S/K_\infty, M)=0$ and hence we find that $H^2(K_S/K_\infty, M[\varpi])^\vee$ is a submodule of $H^1(K_S/K_\infty, M)^\vee$. It follows from \cite[Theorem 4.7]{ochivenjakob} that $H^1(K_S/K_\infty, M)^\vee$ contains no non-zero finite submodules.
\end{proof}
\noindent We define the residual fine Selmer group as follows:
\[\mathcal{R}(\mathbf{V}_{\bar{\rho}}/K_\infty):= \op{ker}\left\{H^1(K_S/K_\infty, \mathbf{V}_{\bar{\rho}})\longrightarrow \bigoplus_{v\in S} K^1_v(\mathbf{V}_{\bar{\rho}}/K_\infty)\right\}.\]
The Poitou--Tate sequence gives us a four term exact sequence:
\begin{equation}\label{PTeqn}0\rightarrow H^0(K_\infty,\mathbf{V}_{\bar{\rho}})\rightarrow \bigoplus_{v\in S} K_v^0(\mathbf{V}_{\bar{\rho}}/K_\infty)\rightarrow H^2_{\op{Iw}}(K_S/K_\infty, \mathbf{V}_{\bar{\rho}^*})^\vee\rightarrow \mathcal{R}(\mathbf{V}_{\bar{\rho}}/K_\infty)\rightarrow 0,\end{equation} cf. \cite[(43)]{CoatesSujatha}. 

    \begin{theorem}\label{Conj A criterion}
Let $K_\infty/K$ be a $\Z_p$-extension such that all primes $v\in S_p(K)$ are totally ramified in $K_\infty$. Let $S$ be a finite set of primes of $K$ containing $S_p(K)$ such that all primes $v\in S$ are finitely decomposed in $K_\infty$. Let $\cO$ be the valuation ring of a $p$-adic field and let $\rho:\op{G}_{K, S}\rightarrow \op{GL}_d(\cO)$ be a continuous Galois representation and let $\bar{\rho}$ be its residual representation. Then, the following are equivalent.
\begin{enumerate}
    \item\label{c1 Conj A criterion} The fine Selmer group $\Rfine{\mathbf{A}(\rho)}{K_\infty}$ is cotorsion over $\Lambda$ and $\mu_\rho=0$.
    \item\label{c2 Conj A criterion} The Iwasawa cohomology group $H^2_{\op{Iw}}(K_S/K_\infty, \mathbf{V}_{\bar{\rho}^*})$ is finite.
            \item $H^2(K_S/K_\infty, \mathbf{V}_{\bar{\rho}^*})$ is finite.
\end{enumerate}

\end{theorem}

\begin{proof}
    Let $\varpi$ be a uniformizer of $\cO$. We note that (1) holds if and only if 
    $\Rfine{\mathbf{A}(\rho)}{K_\infty}[\varpi]$ is finite. We identify $\mathbf{V}_{\bar{\rho}}$ with $\mathbf{A}(\rho)[\varpi]$ and consider the Kummer sequence:
    \[0\rightarrow \mathbf{V}_{\bar{\rho}}\rightarrow \mathbf{A}(\rho)\xrightarrow{\times \varpi} \mathbf{A}(\rho)\rightarrow 0\] and the associated short exact sequence of cohomology groups:
    \[0\rightarrow \frac{H^0(K_\infty, \mathbf{A}(\rho))}{\varpi H^0(K_\infty, \mathbf{A}(\rho))}\xrightarrow{\iota} H^1(K_S/K_\infty, \mathbf{V}_{\bar{\rho}})\xrightarrow{\beta} H^1(K_S/K_\infty, \mathbf{A}(\rho))[\varpi]\rightarrow 0.\]
    We note that kernel of $\beta$ is clearly finite. 
    \par For $w\in S(K_\infty)$, let $\gamma_w$ be the map 
    \[\gamma_w: H^1(K_{\infty, w}, \mathbf{V}_{\bar{\rho}})\rightarrow H^1(K_{\infty, w}, \mathbf{A}(\rho))[\varpi]\] induced from the Kummer sequence. Let \[\gamma_v: K_v^1(\mathbf{V}_{\bar{\rho}}/K_\infty)\rightarrow K_v^1(\mathbf{A}(\rho)/K_\infty)[\varpi]\] be the direct sum of maps $\gamma_w$ as $w$ ranges over the primes of $K_\infty$ which lie above $v$. This set of primes is finite by assumption and it follows that for $v\in S$, the kernel of $\gamma_v$ is finite. Let $\gamma:=\bigoplus_{v\in S} \gamma_v$ be the map 
    \[\gamma: \bigoplus_{v\in S} K_v^1(\mathbf{V}_{\bar{\rho}}/K_\infty)\rightarrow \bigoplus_{v\in S} K_v^1(\mathbf{A}(\rho)/K_\infty)[\varpi].\]
    \noindent Consider the commutative diagram:
\begin{equation}\label{fdiagram}
\begin{tikzcd}[column sep = small, row sep = large]
0\arrow{r} & \mathcal{R}(\mathbf{V}_{\bar{\rho}}/K_\infty)\arrow{r}\arrow{d}{\alpha} & H^1(K_S/K_\infty, \mathbf{V}_{\bar{\rho}})\arrow{r} \arrow{d}{\beta} & \bigoplus_{v\in S} K_v^1(\mathbf{V}_{\bar{\rho}}/K_\infty) \arrow{d}{\gamma} \\
0\arrow{r} & \mathcal{R}_{p^\infty}(\mathbf{A}(\rho)/K_\infty)[\varpi] \arrow{r} & H^1(K_S/K_\infty, \mathbf{A}(\rho))[\varpi] \arrow{r}  & \bigoplus_{v\in S} K_v^1(\mathbf{A}(\rho)/K_\infty)[\varpi].
\end{tikzcd}
\end{equation}
Since the kernels of $\beta$ and $\gamma$ are finite, it follows that the kernel and cokernel of $\alpha$ are both finite. Thus we have shown that $\mathcal{R}_{p^\infty}(\mathbf{A}(\rho)/K_\infty)[\varpi]$ is finite if and only if $\mathcal{R}(\mathbf{V}_{\bar{\rho}}/K_\infty)$ is finite. Thus, (1) is equivalent to $\mathcal{R}(\mathbf{V}_{\bar{\rho}}/K_\infty)$ being finite. On the other hand, it follows from the exact sequence \eqref{PTeqn} that $\mathcal{R}(\mathbf{V}_{\bar{\rho}}/K_\infty)$ is finite if and only if (2) holds. Thus we have shown that (1) and (2) are equivalent. The equivalence (2) and (3) follows from the argument in the proof of \cite[Theorem 3.2]{DRS}.
\end{proof}

\begin{corollary}\label{obviouscor}
    Let $\rho$ be as in Theorem \ref{Conj A criterion} and assume that $H^2(K_S/K_n, \mathbf{V}_{\bar{\rho}^*})=0$ for some value of $n\geq 0$. Then, we have that $\Rfine{\mathbf{A}(\rho)}{K_\infty}$ is cotorsion over $\Lambda$ and $\mu_\rho=0$.
\end{corollary}
\begin{proof}
    Since $\Gamma^{p^n}=\op{Gal}(K_\infty/K_n)$ has $p$ cohomological dimension $1$, the implication
    \[H^2(K_S/K_n, \mathbf{V}_{\bar{\rho}^*})=0\Rightarrow H^2(K_S/K_\infty, \mathbf{V}_{\bar{\rho}^*})^{\op{Gal}(K_\infty/K_n)}=0\] holds (cf. \cite[Exercise 4 (ii), p. 119]{NSW}). Since $\op{Gal}(K_\infty/K_n)$ is a pro-$p$ group, we have the following implication:
    \[H^2(K_S/K_\infty, \mathbf{V}_{\bar{\rho}^*})^{\op{Gal}(K_\infty/K_n)}=0\Rightarrow H^2(K_S/K_\infty, \mathbf{V}_{\bar{\rho}^*})=0.\]
    Finally, by Theorem \ref{Conj A criterion}, $H^2(K_S/K_\infty, \mathbf{V}_{\bar{\rho}^*})=0$ implies that $\Rfine{\mathbf{A}(\rho)}{K_\infty}$ is cotorsion over $\Lambda$ and $\mu_\rho=0$.
\end{proof}

\begin{theorem}\label{secondthm}
    Let $K_\infty$ and $\rho$ be as in Theorem \ref{Conj A criterion} and assume that the weak Leopoldt conjecture holds for $\rho^*$ over $K_\infty$. Then the following are equivalent:
    \begin{enumerate}
        \item The fine Selmer group $\Rfine{\mathbf{A}(\rho)}{K_\infty}$ is cotorsion over $\Lambda$ and $\mu_\rho=0$,
        \item $H^2(K_S/K_\infty, \mathbf{V}_{\bar{\rho}^*})=0$.
    \end{enumerate}
\end{theorem}

\begin{proof}
    It follows from Theorem \ref{Conj A criterion} that (1) is equivalent to the finiteness of $H^2(K_S/K_\infty, \mathbf{V}_{\bar{\rho}^*})$. On the other hand, Proposition \ref{H^2 equivalence propn} asserts that (2) is also equivalent to the finiteness of $H^2(K_S/K_\infty, \mathbf{V}_{\bar{\rho}^*})$. Thus, (1) and (2) are equivalent.
\end{proof}

\begin{proposition}
    Assume that the conditions of Theorem \ref{Conj A criterion} are satisfied. Then we have that 
    \[\begin{split} & \dim H^2_{\op{Iw}}(K_S/K_\infty, \mathbf{V}_{\bar{\rho}^*})\\ =&\dim \mathcal{R}(\mathbf{V}_{\bar{\rho}}/K_\infty)+\sum_{v\in S} \dim K_v^0(\mathbf{V}_{\bar{\rho}}/K_\infty)-\dim H^0(K_\infty, \mathbf{V}_{\bar{\rho}}/K_\infty).\end{split}\]
\end{proposition}
\begin{proof}The result follows from the exactness of \eqref{PTeqn}.
\end{proof}
\begin{remark}\label{remark on rho=1} Consider the special case in which $\bar{\rho}$ is the trivial $1$-dimensional representation. Denote by $\mathcal{F}_S$ the maximal pro-$p$ abelian extension of $K_\infty$ in which all primes of $K_\infty$ are unramified and the primes $w\in S(K_\infty)$ are completely split. Set $\mathcal{Y}_S$ to denote the Galois group $\op{Gal}(\mathcal{F}_S/K_\infty)$ and note that the $\mu$-invariant of $\mathcal{Y}_S$ vanishes if and only if $\mathcal{Y}_S$ is finitely generated as a $\Z_p$-module. Moreover, if $\mu(\mathcal{Y}_S)=0$, then 
\[\lambda(\mathcal{Y}_S)=\op{rank}_{\Z_p} \mathcal{Y}_S\leq \dim (\mathcal{Y}_S\otimes \F_p)=\dim \mathcal{R}(\F_p/K_\infty).\]
Moreover we have equality if and only if $\mathcal{Y}_S$ has no nontrivial finite submodules. If $K_\infty=K_{\op{cyc}}$ is the cyclotomic $\Z_p$-extension of $K$, then the above result states that 
 \[\dim H^2_{\op{Iw}}(K_S/K_\infty, \mu_p)=\dim \mathcal{R}(\F_p/K_\infty)+|S|-1.\]
 \end{remark}

 \section{Galois deformations and unobstructedness}
In this section, we explore the interplay between the deformation theory of Galois representations considered over a $\Z_p$-extension and the structure of fine Selmer groups introduced in the previous section. Let $\mathbb{F}_q$ be a finite field of characteristic $p$, where $q = p^k$ for some integer $k \geq 1$, and let $\mathcal{O} := \operatorname{W}(\mathbb{F}_q)$ denote the ring of Witt vectors over $\mathbb{F}_q$. The ring $\mathcal{O}$ is a complete discrete valuation ring with uniformizer $p$ and residue field $\mathbb{F}_q$.

Let $G$ be a profinite group, and let $N$ be a closed normal subgroup of $G$ such that the quotient group $\Gamma := G/N$ is isomorphic to $\mathbb{Z}_p$, the additive group of $p$-adic integers. For each integer $m \geq 0$, define the finite quotient $\Gamma_m := \Gamma / \Gamma^{p^m} \cong \mathbb{Z}/p^m\mathbb{Z}$. Let $G_m \subseteq G$ denote the preimage of $\Gamma^{p^m}$ under the quotient map $G \twoheadrightarrow \Gamma$, so that $G/G_m \cong \Gamma_m$. In particular, we set $G_0 := G$.

\par Let $V$ be a finite-dimensional $\mathbb{F}_q$-vector space of dimension $d$, and let
$$
\bar{r} : G \to \operatorname{GL}(V) \cong \operatorname{GL}_d(\mathbb{F}_q)
$$
\noindent be a representation. For each $m \geq 0$, we write $\bar{r}_m := \bar{r}|_{G_m}$ for the restriction of $\bar{r}$ to the open subgroup $G_m \subseteq G$, and similarly define $\bar{r}_\infty := \bar{r}|_N$ to be the restriction of $\bar{r}$ to $N$. Let $\mathcal{C}_{\mathcal{O}}$ denote the category whose objects are Artinian local $\mathcal{O}$-algebras $(R, \mathfrak{m}_R)$ with residue field $R/\mathfrak{m}_R \cong \mathbb{F}_q$, and whose morphisms are local homomorphisms of $\mathcal{O}$-algebras. For each such $R$, there is a unique $\mathcal{O}$-algebra surjection $R \twoheadrightarrow \mathbb{F}_q$ compatible with the given structure. We also consider the larger category $\operatorname{CL}_{\mathcal{O}}$ of complete local $\mathcal{O}$-algebras $(R, \mathfrak{m}_R)$ with residue field $\mathbb{F}_q$, where morphisms are continuous local $\mathcal{O}$-algebra homomorphisms. Let $\operatorname{CNL}_{\mathcal{O}}$ be the subcategory of $R\in \op{CL}_{\cO}$ which are Noetherian.

\par We now define what it means to lift or deform the residual representation $\bar{r}$.

\begin{definition}
Let $(R, \mathfrak{m}_R) \in \op{CL}_{\mathcal{O}}$. A \emph{lift} of $\bar{r}$ to $R$ is a continuous representation
$r_R : G \to \operatorname{GL}_d(R)
$ such that the reduction $\bar{r}_R := r_R \mod \mathfrak{m}_R$ equals $\bar{r}$. Two lifts $r$ and $r'$ are said to be \emph{strictly equivalent} if there exists an element $g \in \ker\left(\operatorname{GL}_d(R) \to \operatorname{GL}_d(\mathbb{F}_q)\right)$ such that $r' = g r g^{-1}$. An \emph{$R$-deformation} of $\bar{r}$ is a strict equivalence class of lifts to $\op{GL}_d(R)$.
\end{definition}

\noindent Let $\operatorname{Def}_{\bar{r}}(R)$ be the set of deformations of $\bar{r}$ to $R$. Varying $R$ over $\mathcal{C}_{\mathcal{O}}$ defines a functor $\operatorname{Def}_{\bar{r}} : \mathcal{C}_{\mathcal{O}} \to \operatorname{Sets}$, which assigns to each Artinian local $\mathcal{O}$-algebra $R$ the set of deformations of $\bar{r}$ to $R$. Given a morphism $\varphi: R_1 \to R_2$ in $\mathcal{C}_{\mathcal{O}}$, and a deformation $r_1 : G \to \operatorname{GL}_d(R_1)$ of $\bar{r}$, we obtain a deformation \[r_2 := \op{GL}_d(\varphi) \circ r_1:G\rightarrow \op{GL}_d(R_2),\] where $\op{GL}_d(\varphi): \op{GL}_d(R_1)\rightarrow \op{GL}_d(R_2)$ is obtained by applying $\varphi$ entry-wise to the matrix entries of $r_1$. The condition that $\varphi$ is a local $\mathcal{O}$-algebra homomorphism ensures that this construction respects strict equivalence and compatibility with the residual representation. Likewise there are deformation functors associated to $\bar{r}_m$ and $\bar{r}_\infty$, denoted $\op{Def}_m$ and $\op{Def}_\infty$ respectively. 

In many cases, the deformation functor $\operatorname{Def}_{m}$ is representable or pro-representable by a complete Noetherian local ring, in the following sense.

\begin{definition}
The deformation functor $\operatorname{Def}_m$ is \emph{pro-representable} if there exists $\mathcal{R}_m\in \op{CNL}_{\cO}$ and a universal deformation \[r_m^{\op{univ}} : G_m \to \operatorname{GL}_d(\mathcal{R}_m)\]
\noindent such that for every $R \in \mathcal{C}_{\mathcal{O}}$, the natural map$$
\operatorname{Hom}_{\operatorname{CNL}_{\mathcal{O}}}(\mathcal{R}_m, R) \longrightarrow \operatorname{Def}_m(R)
$$
\noindent given by composition with $\varrho_m$ is a bijection. Thus, deformations of $\bar{r}_m$ to $R$ correspond precisely to local $\mathcal{O}$-algebra homomorphisms from $\mathcal{R}_m$ to $R$.
\end{definition}
\par Suppose that $\operatorname{Def}_m$ is pro-representable for all $m$, the system of rings $\{\mathcal{R}_m\}_{m \geq 0}$ admits natural transition maps $\mathcal{R}_m \to \mathcal{R}_n$ for $m \geq n$. We may then form the inverse limit $\mathcal{R}_\infty := \varprojlim_n \mathcal{R}_n$. We note  that even if each algebra in the inverse system is Noetherian, the inverse limit $\mathcal{R}_\infty$ need not be.

\begin{proposition}
Assume that each of the deformation functors $\op{Def}_m$ is pro-representable by $\mathcal{R}_m$. Then $\operatorname{Def}_\infty$ is pro-represented by the inverse limit ring $
\mathcal{R}_\infty := \varprojlim_m \mathcal{R}_m
$.
\end{proposition}

\begin{proof}
Consider the exact sequence of profinite groups
\[1 \longrightarrow N \longrightarrow G \longrightarrow \Gamma \longrightarrow 1.\]
\noindent Choose a topological generator $\gamma$ of $\Gamma$, and fix a continuous section (i.e., a splitting of the above exact sequence) $
s: \Gamma \to G,
$ sending $\gamma \mapsto \tilde{\gamma} \in G$. This allows us to view $G$ as a semidirect product $
G \cong N \rtimes \Gamma,
$ with $\Gamma$ acting on $N$ by conjugation. That is, the conjugation action of $\gamma \in \Gamma$ on $N$ is given by $g \mapsto \tilde{\gamma} g \tilde{\gamma}^{-1}$.
\par Let $A$ be an object of $\mathcal{C}_\mathcal{O}$, and let $
\rho: N \to \operatorname{GL}_d(A)
$ be a deformation of $\bar{r}_\infty$. Note that $N_1:=\ker(\rho)$ is an open normal subgroup of $N$, and hence of finite index. 
\par We show that there is a finite index subgroup $\Gamma'$ of $\Gamma$ which normalizes $N_1$. Let $E$ be the fixed field of $N_1$. Thus $E/K_\infty$ is a finite Galois
extension, and
\[
\operatorname{Gal}(E/K_\infty)\simeq N/N_1.
\]
Since $E/K_\infty$ is finite and $K_\infty=\bigcup_m K_m$, there exists an
integer $m$ and a finite extension $E_m/K_m$ such that
\[
E=E_mK_\infty.
\]
Let $\widetilde{E_m}$ be the Galois closure of $E_m$ over $K_m$.
For any element $\eta\in \Gamma^{p^m}=\operatorname{Gal}(K_\infty/K_m)$ and
any lift $\widetilde{\eta}\in G$, the field $\widetilde{\eta}(E)$ is of the
form
\[
\widetilde{\eta}(E)=\widetilde{\eta}(E_m)K_\infty.
\]
As $\widetilde{\eta}$ fixes $K_m$, the field $\widetilde{\eta}(E_m)$ is one
of the finitely many $K_m$-conjugates of $E_m$ contained in
$\widetilde{E_m}$. Hence the orbit
\[
\left\{\widetilde{\eta}(E)\mid \eta\in \Gamma^{p^m}\right\}
\]
is finite. Since $\Gamma^{p^m}$ has finite index in $\Gamma$, it follows that 
\[
\left\{\widetilde{\eta}(E)\mid \eta\in \Gamma\right\}
\]is finite.
Therefore the stabilizer
\[
\Gamma'
:=
\left\{
\eta\in \Gamma
\mid
\widetilde{\eta}(E)=E
\right\}
\]
has finite index in $\Gamma$. Note that for every $\eta\in \Gamma'$, one has
\[
\widetilde{\eta}N_1\widetilde{\eta}^{-1}=N_1.
\]
\par Since $\Gamma\simeq \mathbb Z_p$, there exists an integer
$n_0\geq 0$ such that $\Gamma^{p^{n_0}}= \Gamma'$. Thus $\widetilde{\gamma}^{p^{n_0}}$ acts on the finite quotient $N/N_1$. In particular the image of $\widetilde{\gamma}^{p^{n_0}}$ in $\operatorname{Aut}(N/N_1)$ has finite order. Hence there exists an integer $n\geq n_0$ such that
$\widetilde{\gamma}^{p^n}$ acts trivially on $N/N_1$.

\par Let $G_n$ denote the inverse image in $G$ of $\Gamma^{p^n} \subset \Gamma$, so that $G_n = \langle N, \tilde{\gamma}^{p^n} \rangle$. Let $H_1 := \Gamma^{p^n} \subset \Gamma$. By construction, for every $\sigma \in H_1$ and $g \in N$, we have $
\rho(\sigma g \sigma^{-1}) = \rho(g)
$. This condition ensures that $\rho$ extends in a natural way to a representation of the semidirect product $G_n = N \rtimes H_1$, where we define $
\rho(g \cdot \sigma) := \rho(g)
$ for $g \in N$, $\sigma \in H_1$. Thus, $\rho$ extends to a deformation of $\bar{r}_n$ to $A$, i.e., to an element of $\operatorname{Def}_n(A)$. Recall that $\bar{r}_\infty:=\bar{r}_{|N}$. Thus every deformation of $\bar{r}_\infty$ to $A$ arises as the restriction to $N$ of a deformation of $\bar{r}_n$ for sufficiently large $n$. This implies that there is a natural identification
\[\operatorname{Def}_\infty = \varinjlim_n \operatorname{Def}_n,\]
\noindent where the colimit is taken over the directed system of functors $\operatorname{Def}_n$ as $n \to \infty$. The transition maps in this system are induced by the inclusions $G_n \subseteq G_m$ for $m \geq n$, and hence by the corresponding maps of representing rings $\mathcal{R}_m \to \mathcal{R}_n$.

\par Since each $\operatorname{Def}_n$ is pro-represented by $\mathcal{R}_n$, and limits of functors correspond to inverse limits of pro-representing objects, it follows that the colimit functor $\operatorname{Def}_\infty$ is pro-represented by the inverse limit $
\mathcal{R}_\infty := \varprojlim_n \mathcal{R}_n$. This completes the proof.
\end{proof}

\par For $R\in \op{CL}_{\cO}$ and a continuous representation $r:G\to\op{GL}_d(R)$. The adjoint representation $\operatorname{Ad}r$ is the $R[G]$-module of $d\times d$ matrices with $G$-action $\sigma\cdot M:=r(\sigma) M r(\sigma)^{-1}$. Because scalar matrices are fixed by conjugation, the trace map $\operatorname{tr}:\operatorname{Ad}r\to R$ is $G$-equivariant and splits $\operatorname{Ad}r\cong\operatorname{Ad}^0 r\oplus R \cdot\mathbf{I}_R$, where $\operatorname{Ad}^0 r=\ker\operatorname{tr}$ are the trace zero endomorphisms and $R\cdot\mathbf{I}_R$ is the trivial $1$-dimensional representation spanned by the identity matrix $\mathbf{I}_R$. One has that
$$
H^i(G,\operatorname{Ad}r)\cong H^i(G,\operatorname{Ad}^0 r)\oplus H^i(G,R).
$$
\noindent Consider the ring of dual numbers $D:=\F_q[\varepsilon]=\F_q[x]/(x^2)$, with $\varepsilon^2=0$. An infinitesimal lift of $\bar{r}$ is a lift $r:G\to\op{GL}_d(\F_q[\varepsilon])$ whose reduction modulo $\varepsilon$ is $\bar{r}$. Writing such a lift in the form
$$
r(\sigma)=\bigl(1+\varepsilon\,c(\sigma)\bigr)\,\bar{r}(\sigma)
$$
\noindent for $c(\sigma)\in\op{Ad}\bar{r}$, the condition that $r$ is a homomorphism is equivalent to the cocycle condition
\[c(\sigma\tau)=c(\sigma)+\sigma\cdot c(\tau).\]
This sets up a bijection between infinitesimal lifts of $\bar{r}$ and cocycles $Z^1(G, , \op{Ad}\bar{r})$. This in turn gives a bijection: \[\op{Def}_{\bar{r}}(D)\leftrightarrow H^1(G, \op{Ad}\bar{r}).\]Moreover, deformations with determinant equal to $\op{det}\bar{r}$ are in bijection with $H^1(G, \op{Ad}^0\bar{r})$. Let $R\in \mathcal{C}_{\cO}$ and $R':=R/J$ be a quotient of $R$. The quotient map $R\rightarrow R'$ is said to a small extension if the ideal $J$ is generated by a single element $t$ such that $t\cdot \mathfrak{m}_R=0$. For instance, the quotient map $\cO/p^{m+1}\rightarrow \cO/p^{m}$ is a small extension. Given a small extension $\pi: R\rightarrow R'$, the fibers of the induced map
\[\pi^*: \op{Def}_{\bar{r}}(R)\rightarrow \op{Def}_{\bar{r}}(R')\]are an $H^1(G, \op{Ad}\bar{r})$ pseudo-torsor. In other words, let $r\in \op{Def}_{\bar{r}}(R')$ and suppose that the fiber $(\pi^*)^{-1}(r)$ is nonempty and $r_1, r_2\in (\pi^*)^{-1}(r)$. Then there is a unique cohomology class $c\in H^1(G, \op{Ad}\bar{r})$ such that $r_2=(1+t c) r_1$. Moreover $\op{det}r_1=\op{det}r_2$ if and only if $c\in H^1(G, \op{Ad}^0\bar{r})$. 
\par We record the following standard identification, which will be used in the
obstruction-theoretic discussion below. Let $\pi:R\longrightarrow R'$
be a small extension with kernel \(J=(t)\). If \(A\in M_d(\mathbb F_q)\) and \(\widetilde A\) is any
lift of \(A\) to \(M_d(R)\), then \(t\widetilde A\) is independent of the
choice of lift since \(t\mathfrak m_R=0\). We shall therefore denote it simply by \(tA\). The map
\[
M_d(\mathbb F_q)
\longrightarrow
\ker\left\{
\operatorname{GL}_d(R)\longrightarrow \operatorname{GL}_d(R')
\right\},
\]
defined by $A\longmapsto I+tA$ gives an isomorphism from the additive group \(M_d(\mathbb F_q)\) to the
multiplicative kernel. Indeed,
\[
(I+tA)(I+tB)=I+t(A+B),
\]
since \(t^2=0\). Under this identification, conjugation by a lift of
\(\bar r(\sigma)\) corresponds to the usual adjoint action of \(\sigma\) on
\(\operatorname{Ad}\bar r\). We shall therefore identify the above kernel with
\(\operatorname{Ad}\bar r\).

\par We now recall the obstruction class attached to a small extension. Let
\(r\in \operatorname{Def}_{\bar r}(R')\), and choose a set-theoretic lift
\[
\widetilde r:G\longrightarrow \operatorname{GL}_d(R)
\]
of \(r\). This lift need not be a homomorphism. The defect of multiplicativity
is measured by
\[
c_{\widetilde r}(\sigma,\tau)
:=
\widetilde r(\sigma)\widetilde r(\tau)\widetilde r(\sigma\tau)^{-1}.
\]
Since \(r\) is a homomorphism modulo \(J\), the element
\(c_{\widetilde r}(\sigma,\tau)\) lies in
\[
\ker\left\{
\operatorname{GL}_d(R)\longrightarrow \operatorname{GL}_d(R')
\right\}.
\]
Thus there is a uniquely determined map
\[
o_{\widetilde r}:G\times G\longrightarrow \operatorname{Ad}\bar r
\]
such that
\[
c_{\widetilde r}(\sigma,\tau)=I+t\,o_{\widetilde r}(\sigma,\tau).
\]
The map \(o_{\widetilde r}\) is a continuous \(2\)-cocycle, and its cohomology
class
\[
\mathcal O(r):=[o_{\widetilde r}]\in H^2(G,\operatorname{Ad}\bar r)
\]
is independent of the choice of the set-theoretic lift \(\widetilde r\). This
class is the obstruction to lifting \(r\) across \(\pi\): it vanishes if and
only if \(r\) admits a lift to \(R\) which is a group homomorphism.

\begin{definition}The deformation functor $\op{Def}_n$ (resp. $\op{Def}_\infty$) is \emph{unobstructed} if $H^2(G_n, \op{Ad}\bar{r}_n)=0$ (resp. $H^2(N, \op{Ad}\bar{r}_\infty)=0$). 
\end{definition}

\noindent We note that the above definition makes sense even when the deformation functors $\op{Def}_n$ are not known to be pro-representable. 
\begin{proposition}
    Suppose that $\bar{r}_n$ is unobstructed for some $n\geq 0$, then it follows that $\bar{r}_\infty$ is also unobstructed. 
\end{proposition}
\begin{proof}
    The result follows from the proof of Corollary \ref{obviouscor}.
\end{proof}

\par Now let $K$ be a number field and $K_\infty$ be a $\Z_p$-extension of $K$ in which all primes $v\in S_p(K)$ are totally ramified. Let $S$ be a finite set of primes containing $S_p(K)$ and assume that all primes $v\in S$ are finitely decomposed in $K_\infty$. Let $r:\op{G}_{K, S}\rightarrow \op{GL}_d(\cO)$ be a continuous Galois representation and $\rho:=(\op{Ad} r)(1)$. Observe that $\rho^*=\op{Ad}r$. We shall now set $G:=\op{Gal}(K_S/K)$ and $N:=\op{Gal}(K_S/K_\infty)$ and $\Gamma=G/N$. 

\begin{proposition}\label{unobstructedness for def rings}
 Suppose that the weak Leopoldt conjecture holds for $\rho=\op{Ad}r(1)$. Then the following are equivalent:
    \begin{enumerate}
        \item $\bar{r}_{\infty}$ is unobstructed, 
        \item the fine Selmer group $\mathcal{R}_{p^\infty}(A(\rho)/K_\infty)$ is cotorsion over $\Lambda$ and $\mu_\rho=0$.
    \end{enumerate}
\end{proposition}
\begin{proof}
    As a consequence of Theorem \ref{secondthm}, condition (2) is equivalent to the vanishing of $H^2(K_S/K_\infty, \mathbf{V}_{\bar{\rho}^*})$. Since $\rho^*=\op{Ad}r$, (2) is equivalent to (1).
\end{proof}

\section{Ordinary deformations over the $\Z_p$-extension}
\par In this section, we refine certain results in \cite{BCM1} concerning the unobstructedness of ordinary deformation rings in the cyclotomic limit. The main theorem in \emph{loc.\ cit.} establishes unobstructedness under two separate hypotheses: first, that the deformation ring in the cyclotomic limit is Noetherian, and second, that the $\mu$-invariant of the relevant Selmer group vanishes. We prove that the Noetherian property is equivalent to the vanishing of the $\mu$-invariant of the Greenberg Selmer group of the adjoint representation. In fact, the tangent space of the deformation ring over the cyclotomic $\Z_p$-extension can be naturally identified with the residual Selmer group of the adjoint representation, which is finite if and only if the $\mu$-invariant vanishes. This result fits into the broader theme of the article, since it elucidates the relationship between Galois deformation rings over $\Z_p$-extensions and the Iwasawa theory of Selmer groups.

\par Throughout this section, $K$ will be a totally real number field and $p$ will denote an odd prime number which is unramified in $K$. Let $v\in S_p(K)$ be a prime of $K$. Denote by $K_{\op{cyc}}/K$ the cyclotomic $\Z_p$-extension of $K$ and let $K_n\subset K_{\op{cyc}}$ be its $n$-th layer. Note that $v$ is totally ramified in $K_{\op{cyc}}$. We fix a finite set of primes $S$ containing the primes of $K$ which lie above $p$. Let $\F_q$ be a finite field of characteristic $p$ and let $\cO$ be the valuation ring of a $p$-adic field with residue field isomorphic to $\F_q$. Let $\op{W}(\F_q)$ denote the ring of Witt vectors with residue field $\F_q$. Fix a uniformizer $\varpi$ of $\cO$. We let $r: \op{G}_{K, S}\rightarrow \op{GL}_d(\cO)$ be a continuous representation and let $\bar{r}:=r\pmod{(\varpi)}$ denote its residual representation. Assume that $\det r$ is the $p$-adic cyclotomic character $\chi$. 
\begin{definition}
    Let $r$ be a deformation of $\bar{r}$ and $v\in S_p(K)$. Then $\varrho$ is said to be ordinary of weight $2$ at $v$ if the restriction of $r$ to $\op{G}_{K_{v}}$ is isomorphic to $\mtx{\chi \alpha}{\ast}{0}{\alpha^{-1}}$ where $\alpha$ is a finite order unramified character. We say that $r$ is ordinary of weight $2$ if it is ordinary of weight $2$ at each prime $v\in S_p(K)$. 
\end{definition}
\noindent Throughout we make the following assumptions:
\begin{enumerate}
    \item $r$ (and thus $\bar{r}$) is ordinary of weight $2$,
    \item $\bar{r}_{|\op{G}_{K(\mu_p)}}$ is irreducible, 
    \item the restriction of $\bar{r}$ to each prime $v\in S_p(K)$ is absolutely indecomposable.
\end{enumerate}
\noindent Let $R_n$ be the deformation ring parameterizing ordinary liftings of $\bar{r}_{|\op{G}_{K_n}}$ over $\op{W}(\F_q)$.

\begin{proposition}
    Under the conditions above, $R_n$ belongs to $\op{CNL}_{\op{W}(\F_q)}$ and the natural map $R_{n+1}\rightarrow R_n$ is surjective for all $n$.
\end{proposition}

\begin{proof}
    The result follows from \cite[Theorem 1.1 and Lemma 1.2]{BCM1}.
\end{proof}

\begin{proposition}
     The inverse limit $R_\infty:=\varprojlim_n R_n$ represents the functor of ordinary deformations of $\bar{r}_{|\op{G}_{F_{\op{cyc}}}}$.
\end{proposition}
\begin{proof}
    See \cite[Corollary 1.5]{BCM1}.
\end{proof}
\par Let $\rho:=\op{Ad}r$ be the adjoint representation associated to $r$. We point out that the notation here is different from that in the previous section, where $\rho$ was taken to be Tate twist of $\op{Ad}r$. For $v\in S_p(K)$, since $\rho$ is $v$-ordinary, there is a short exact sequence of $\op{G}_{K_{v}}$-modules:
\[0\rightarrow \mathbf{A}(\rho)^+\rightarrow \mathbf{A}(\rho)\rightarrow \mathbf{A}(\rho)^-\rightarrow 0\] such that the inertia group at $v$ acts on the semi-simplification of $\mathbf{A}(\rho)^+$ (resp. $\mathbf{A}(\rho)^-$) via positive (resp. non-positive) powers of the cyclotomic character. Then the ordinary Selmer group $\op{Sel}_{\op{Gr}}(\mathbf{A}(\rho)/K_{\op{cyc}})$ consists of cohomology classes $f\in H^1(K_S/K_{\op{cyc}}, \mathbf{A}(\rho))$ which are unramified at all primes $v\nmid p$ and ordinary at the primes $v|p$ of $\op{K}_{\op{cyc}}$. In greater detail, given a prime $v\in \Omega_K$ such that $v\nmid p$, let 
\[\cH_v(\mathbf{A}(\rho)/K_{\op{cyc}}):=\bigoplus_{w|v} H^1_{\op{nr}}(K_{\op{cyc}, w}, \mathbf{A}(\rho))\] where $w$ ranges over the primes of $K_{\op{cyc}}$ which lie above $v$. On the other hand, for a prime $w|p$ of $K_{\op{cyc}}$, let $H^1_{\op{ord}}(K_{\op{cyc}, w}, \mathbf{A}(\rho))$ consist of the cohomology classes which give unramified classes on $\mathbf{A}(\rho)^-$. For $v|p$, set 
\[\cH_v(\mathbf{A}(\rho)/K_{\op{cyc}}):=H^1_{\op{ord}}(K_{\op{cyc}, w_v}, \mathbf{A}(\rho))\] where $w_v$ is the unique prime of $K_{\op{cyc}}$ which lies above $v$. 
\begin{definition}
    With respect to notation above, the Greenberg Selmer group associated to $\mathbf{A}(\rho)$ is defined as follows:
    \[\op{Sel}_{\op{Gr}}(\mathbf{A}(\rho)/K_{\op{cyc}}):=\op{ker}\left\{H^1(K_S/K_{\op{cyc}}, \mathbf{A}(\rho))\longrightarrow \bigoplus_{v\in S} \cH_v(\mathbf{A}(\rho)/K_{\op{cyc}})\right\}.\]
\end{definition}
\noindent Likewise, we may define a Greenberg Selmer group associated to the residual representation:
\[\op{Sel}_{\op{Gr}}(\mathbf{V}_{\bar{\rho}}/K_{\op{cyc}}):=\op{ker}\left\{H^1(K_S/K_{\op{cyc}}, \mathbf{V}_{\bar{\rho}})\longrightarrow \bigoplus_{v\in S} \cH_v(\mathbf{V}_{\bar{\rho}}/K_{\op{cyc}})\right\},\]
where
\[\cH_v(\mathbf{V}_{\bar{\rho}}/K_{\op{cyc}}):=\begin{cases}
    \bigoplus_{w|v} H^1_{\op{nr}}(K_{\op{cyc}, w}, \mathbf{V}_{\bar{\rho}}) \quad &\text{ if } v\nmid p,\\
    H^1_{\op{ord}}(K_{\op{cyc}, w_v}, \mathbf{V}_{\bar{\rho}})\quad &\text{ if } v| p.
\end{cases}\]
\noindent Identifying $\mathbf{V}_{\bar{\rho}}$ with $\mathbf{A}(\rho)[\varpi]$, the Kummer sequence 
\[0\rightarrow \mathbf{V}_{\bar{\rho}}\rightarrow \mathbf{A}(\rho)\xrightarrow{\times \varpi}\mathbf{A}(\rho)\rightarrow 0 \]
induces a natural map 
\[\alpha: \op{Sel}_{\op{Gr}}(\mathbf{V}_{\bar{\rho}}/K_{\op{cyc}})\longrightarrow \op{Sel}_{\op{Gr}}(\mathbf{A}(\rho)/K_{\op{cyc}})[\varpi].\] 
\begin{lemma}\label{alpha lemma}
    The map $\alpha$ has finite kernel and cokernel.
\end{lemma}
\begin{proof}
    The map $\alpha$ fits into a commutative diagram:
    \begin{equation}\label{fdiagram}
\begin{tikzcd}[column sep = small, row sep = large]
0\arrow{r} & \op{Sel}_{\op{Gr}}(\mathbf{V}_{\bar{\rho}}/K_{\op{cyc}}) \arrow{r}\arrow{d}{\alpha} & H^1(K_S/K_{\op{cyc}}, \mathbf{V}_{\bar{\rho}})\arrow{r} \arrow{d}{\beta} & \bigoplus_{v\in S} \cH_v(\mathbf{V}_{\bar{\rho}}/K_{\op{cyc}}) \arrow{d}{\gamma}  \\
0\arrow{r} & \op{Sel}_{\op{Gr}}(\mathbf{A}(\rho)/K_{\op{cyc}})[\varpi]\arrow{r} & H^1(K_S/K_{\op{cyc}}, \mathbf{A}(\rho))[\varpi]\arrow{r}  &\bigoplus_{v\in S} \cH_v(\mathbf{A}(\rho)/K_{\op{cyc}})[\varpi],
\end{tikzcd}
\end{equation}
where the map $\gamma$ is a direct sum of maps 
\[\gamma_v: \cH_v(\mathbf{V}_{\bar{\rho}}/K_{\op{cyc}})\rightarrow \cH_v(\mathbf{A}(\rho)/K_{\op{cyc}})[\varpi].\]
The map $\beta$ is surjective and its kernel is isomorphic to 
\[\frac{H^0(K_{\op{cyc}}, A(\rho))}{\varpi H^0(K_{\op{cyc}}, A(\rho))}.\] In particular, the kernel of $\beta$ is finite. It follows that the kernel of $\alpha$ is also finite.
It is easy to see that each of the maps $\gamma_v$ has finite kernel. The cokernel of $\alpha$ is contained in the kernel of $\gamma$, which is finite.
\end{proof}

\begin{proposition}\label{mainpropsection5}
    With respect to notation above, the following conditions are equivalent:
    \begin{enumerate}
        \item The ordinary deformation ring $R_\infty$ is Noetherian.
        \item The residual Selmer group $\op{Sel}_{\op{Gr}}(\mathbf{V}_{\bar{\rho}}/K_{\op{cyc}})$ is finite. 
        \item The ordinary Selmer group $\op{Sel}_{\op{Gr}}(\mathbf{A}(\rho)/K_{\op{cyc}})$ is a cotorsion $\Lambda$-module and its $\mu$ invariant vanishes.
    \end{enumerate}
\end{proposition}
\begin{proof}
    The equivalence of (1) and (2) follows since $\op{Sel}_{\op{Gr}}(\mathbf{V}_{\bar{\rho}}/K_{\op{cyc}})$ can be identified with the mod-$p$ tangent space of $R_\infty$ (see the proof of \cite[Theorem 5.3]{BCM1}). From Lemma \ref{alpha lemma}, the condition (2) is equivalent to the finiteness of $\op{Sel}_{\op{Gr}}(\mathbf{A}(\rho)/K_{\op{cyc}})[\varpi]$. This is in turn equivalent to condition (3).
\end{proof}

\par Let us now recall the main result of \cite{BCM1}, which asserts that $R_\infty$ is formally smooth provided it is Noetherian and some further conditions are satisfied. We say that $\bar{\rho}$ is \emph{adequate} if $\bar{\rho}(\op{G}_{K(\mu_p)})$ is an adequate subgroup in the sense of \cite[Definition 3.1.1]{PAllen}.

\begin{theorem}[Burungale--Clozel]\label{BC1thm}
    Assume that $\rho$ is automorphic and adequate. Further assume that the $\mu$-invariants of $H^1(K_S/K_{\op{cyc}}, A(\rho))_{\op{tors}}$ and $H^1(K_S/K_{\op{cyc}}, A(\rho^*))_{\op{tors}}$ both vanish. Then if $R_\infty$ is Noetherian, then it is formally smooth, i.e.,  
    \[R_\infty\simeq \op{W}(\F_q)\llbracket X_1, \dots, X_s\rrbracket.\]
\end{theorem}
\noindent As a consequence of Proposition \ref{mainpropsection5}, one obtains the following result.
\begin{corollary}\label{cor5.8}
    Let $\rho$ be as in Theorem \ref{BC1thm} and let $r:=\op{Ad}\rho$. Then the following are equivalent:
    \begin{enumerate}
        \item The ordinary Selmer group $\op{Sel}_{\op{Gr}}(\mathbf{A}(\rho)/K_{\op{cyc}})$ is a cotorsion $\Lambda$-module and its $\mu$ invariant vanishes.
        \item $R_\infty$ is a power series ring $\op{W}(\F_q)\llbracket X_1, \dots, X_s\rrbracket$.
    \end{enumerate}
\end{corollary}
\begin{proof}
    That the condition (1) is equivalent to the condition that $R_\infty$ is Noetherian is the assertion of Proposition \ref{mainpropsection5}. The condition (2) then follows from Theorem \ref{BC1thm}. This shows that (1) implies (2). The condition (2) implies that $R_\infty$ is in particular, Noetherian and thus by Proposition \ref{mainpropsection5}, condition (1) follows from (2). We have thus shown that (1) and (2) are equivalent.
\end{proof}

\noindent The unobstructedness of the deformation ring follows from Theorem \ref{BC1thm} once the Noetherian property is established. On the other hand, Proposition \ref{unobstructedness for def rings} from the previous section is a condition for unobstructedness which does not guarantee that the deformation ring $R_\infty$ is Noetherian.

\section{Bockstein maps and Massey products}
\par In the next section, we establish the main theorems of the paper, demonstrating that the vanishing of the $\mu$-invariant of the fine Selmer group is equivalent to a criterion formulated in terms of higher Massey products. Our strategy is to translate the Iwasawa-theoretic problem into a question about the structure of certain Galois cohomology groups, where Massey products naturally arise. This is made possible by the work of Lam, Liu, Sharifi, Wake, and Wang \cite{Lametal}, who analyzed the fine structure of Iwasawa cohomology groups and clarified their connections with Bockstein maps and higher Massey products. In this section, we set up relevant notation (consistent with \emph{loc. cit.}) and recall some important results in this context.
\subsection{Bockstein-maps}
\par Let $G$ be a profinite group with cohomological dimension $d$ and $H$ be a pro-$p$ quotient of $G$ by a closed normal subgroup $N$. Assume throughout that $G$ has finite $p$-cohomological dimension. Define the completed group algebra
$$
\Omega := \F_p\llbracket H \rrbracket = \varprojlim_{U} \F_p[H/U],
$$
\noindent where the inverse limit ranges over all open normal subgroups $U$ of finite index in $H$. Let $V$ be an $\F_p\llbracket G \rrbracket$-module that is finite dimensional as an $\F_p$ vector space. Let $\alpha: \Omega \to \F_p$ be the augmentation map sending each group-like element $g \in H$ to $1 \in \F_p$, and let $I := \ker(\alpha)$ denote the \emph{augmentation ideal} of $\Omega$.

\par The short exact sequence of $\F_p \llbracket G \rrbracket$-modules
\[
0 \rightarrow V \otimes_{\F_p} I^{n} / I^{n+1} \rightarrow V \otimes_{\F_p} \Omega / I^{n+1} \rightarrow V\otimes_{\F_p} \Omega / I^{n} \rightarrow 0 \]
\noindent induces the connecting homomorphism
\[H^{d-1}\left(G, V\otimes_{\F_p} \Omega / I^{n}\right) \longrightarrow H^{d}\left(G, V\otimes_{\F_p} I^{n} / I^{n+1}\right)\] on passing to cohomology. Since $G$ acts trivially on $I^{n} / I^{n+1}$, there is a natural isomorphism
\[
H^{d}(G, V) \otimes_{\F_p} I^{n} / I^{n+1} \xrightarrow{\sim} H^{d}\left(G, V \otimes_{\F_p} I^{n} / I^{n+1}\right) .\]
\begin{definition}Define
\[\Psi^{(n)}: H^{d-1}\left(G, V \otimes_{\F_p} \Omega / I^{n}\right) \rightarrow H^{d}(G, V) \otimes_{\F_p} I^{n} / I^{n+1}\]
\noindent to be the resulting composite map, and refer to it as the \emph{generalized Bockstein map}.
\end{definition}
\noindent Recall that, given a short exact sequence of $G$-modules
\[
0 \to A \to B \to C \to 0,
\]
there is a connecting homomorphism
\[
\beta \colon H^i(G,C) \to H^{i+1}(G,A),
\]
called the \emph{Bockstein homomorphism}. The above map arises as the connecting homomorphism in cohomology associated to a short exact sequence of coefficients, and hence is a natural generalization of the classical Bockstein homomorphism.
\noindent For $i>0$, define the $i$-th \emph{Iwasawa cohomology group} as follows:
\[H^i_{\op{Iw}}(N, V):=\varprojlim_U H^i(U, V)\]where $U$ ranges over open normal subgroups of $G$ which contain $N$. 
\begin{theorem}[Lam,\,Liu,\,Sharifi,\,Wake\,and\,Wang]
    For each $n \geq 1$, there is a canonical isomorphism
\begin{equation}\label{LLSWW eqn}
\frac{I^{n} H_{\op{Iw}}^{d}(N, V)}{I^{n+1} H_{\op{Iw}}^{d}(N, V)} \cong \frac{H^{d}(G, V) \otimes_{\F_p} I^{n} / I^{n+1}}{\operatorname{im} \Psi^{(n)}}
\end{equation}
of $R$-modules, where $d$ is the $p$-cohomological dimension of $G$.
\end{theorem}

\begin{proof}
    See \cite[Theorem 2.2.4]{Lametal}.
\end{proof}

\begin{remark}
    In our applications, $K$ will be a number field and $S
    \subset \Omega_K$ is a finite set of primes containing all primes of $K$ which lie above $p$. Let $K_\infty/K$ be a $\Z_p$-extension such that:
    \begin{itemize}
        \item all primes $v\in S_p(K)$ are totally ramified in $K_\infty$,
    \item all primes $v\in S$ are finitely decomposed in $K_\infty$. \end{itemize}Setting 
    \[G:=\op{Gal}(K_S/K)\text{ and }N:=\op{Gal}(K_S/K_\infty),\] we find that $H=\Gamma=\op{Gal}(K_\infty/K)$ and $\Omega=\Lambda/(p)=\F_p\llbracket T\rrbracket $. Note that $I$ is the maximal ideal $(T)$. Let $\rho:\op{G}_{K,S}\rightarrow \op{GL}_d(\cO)$ be a continuous Galois representation and $V:=\mathbf{V}_{\bar{\rho}^*}$. Then \eqref{LLSWW eqn} specializes to 
    \[\frac{I^{n} H_{\op{Iw}}^{2}(K_S/K_\infty, \mathbf{V}_{\bar{\rho}^*})}{I^{n+1} H_{\op{Iw}}^{2}(K_S/K_\infty, \mathbf{V}_{\bar{\rho}^*})} \cong \frac{H^{2}(K_S/K, \mathbf{V}_{\bar{\rho}^*}) \otimes_{\F_p} I^{n} / I^{n+1}}{\operatorname{im} \Psi^{(n)}}.\] In particular, $H_{\op{Iw}}^{2}(K_S/K_\infty, \mathbf{V}_{\bar{\rho}^*})$ is finite if and only if there exists $n>0$ such that the image of $\Psi^{(n)}$ spans $H^{2}(K_S/K, \mathbf{V}_{\bar{\rho}^*}) \otimes_{\F_p} I^{n} / I^{n+1}$.
\end{remark}

\par Coming back to the general situation, assume for simplicity that \( H \cong \mathbb{Z}_p^r \). We fix an isomorphism
\begin{equation}\label{Hab iso}
H \cong \bigoplus_{i=1}^{r} A_i,
\end{equation}
where each \( A_i = \mathbb{Z}_p \) for \( i = 1, \ldots, r \). For each \( i \), let \[\chi_i: G \to A_i\]
denote the composition of the canonical projection \( G \twoheadrightarrow H \) with the projection onto the \( i \)-th summand \( H \twoheadrightarrow A_i \). The image of \( \Psi^{(1)} \) may be described in terms of cup products involving these homomorphisms, as we now explain.
\par Choose generators $h_1, \ldots, h_r \in H$ such that the image of $h_i$ in $A_i$ under the isomorphism \eqref{Hab iso} is $1 \in A_i$. Then, setting $x_i = [h_i] - 1 \in I$, where $I$ is the augmentation ideal in $\Omega = \F_p\llbracket H \rrbracket$, we obtain an isomorphism \[I/I^2 \cong H/pH=H\otimes \F_p\] sending the image of $x_i$ to $h_i \otimes 1$.

\begin{proposition}
    With respect to notation above, for any 1-cocycle $f: G \rightarrow V$, we have
$$
\Psi^{(1)}([f]) = \sum_{i=1}^{r} (\chi_i \cup f)\, x_i \in H^2(G, V) \otimes_R I/I^2.
$$
\end{proposition}
\begin{proof}
    This result is a special case of \cite[Proposition 2.3.3]{Lametal}.
\end{proof}
\subsection{Massey products}\label{section 6.2 Massey}
\par For $n \geq 2$, the image of $\Psi^{(n)}$ is described in terms of $(n+1)$-fold \emph{Massey products}, which may be regarded as higher-order generalizations of the cup product in cohomology. These products were first introduced by Massey \cite{Massey} to show that two topological spaces can possess isomorphic cohomology rings without being homotopy equivalent, thereby revealing that the ring structure alone may not capture all homotopical information. In the setting of Galois cohomology, Massey products similarly detect relations among cohomology classes that are invisible to the ordinary cup product, and they encode higher-order extension data.

\par Let \( r \geq 2 \), $R$ be a profinite commutative ring, $G$ a profinite group and let \( T_1, \dots, T_r \) be \( R\llbracket G\rrbracket \)-modules. Given cohomology classes \( \chi_i \in H^1(G, T_i) \) for \( i = 1, \dots, r \), the associated Massey product is an element
\[
(\chi_1, \chi_2, \dots, \chi_r )_\rho \in H^2(G, T_1 \otimes_R T_2 \otimes_R \cdots \otimes_R T_r),
\]
\noindent with respect to a choice of a \emph{defining system} \( \rho \). The construction of this product depends on choices made in specifying \( \rho \), but the resulting ambiguity is controlled in a well-understood manner.

\par We now recall the structure underlying the algebraic framework used to define these products, following the notation and conventions of \cite{Lametal}. Let \( n \geq 1 \) be an integer. In our applications we specialize to the case \( R = \mathbb{F}_p \). An \( n \)-dimensional \emph{upper-triangular generalized matrix algebra} over \( R \), or \( R \)-\emph{UGMA}, is an \( R \)-algebra \( \mathcal{A} \) determined by the following data. For each pair \( 1 \leq i \leq j \leq n \), let \( A_{i,j} \) be a finitely generated \( R \)-module, with \( A_{i,i} = R \) for all \( i \). For each triple \( 1 \leq i \leq j \leq k \leq n \), there are composition maps
\[
\varphi_{i,j,k} : A_{i,j} \otimes_R A_{j,k} \to A_{i,k},
\]
\noindent which are \( R \)-bilinear and agree with scalar multiplication when \( i = j \) or \( j = k \). These maps are required to satisfy an associativity condition: for all \( 1 \leq i < j < k < l \leq n \), the two natural ways of composing
\[
A_{i,j} \otimes_R A_{j,k} \otimes_R A_{k,l} \to A_{i,l}
\]
\noindent must yield the same result. The underlying \( R \)-module of \( \mathcal{A} \) is
\[
\mathcal{A} = \bigoplus_{1 \leq i \leq j \leq n} A_{i,j},
\]
\noindent and multiplication in \( \mathcal{A} \) is defined via the composition maps. For elements \( a = (a_{i,j}) \) and \( b = (b_{i,j}) \) in \( \mathcal{A} \), the \((i,j)\)-entry of the product \( ab \) is given by
\[
(ab)_{i,j} = \sum_{k = i}^j \varphi_{i,k,j}(a_{i,k} \otimes b_{k,j}).
\]

\begin{definition}
Let \( R \) be a profinite ring and \( G \) a profinite group. A \emph{profinite} \( (R, G)\text{-UGMA} \) is a profinite \( R \)-UGMA \( \mathcal{A} \) equipped with a continuous action of \( G \) on each \( A_{i,j} \), such that:
\begin{itemize}
    \item the action on \( A_{i,i} = R \) is trivial for all \( i \), and
    \item each composition map \( \varphi_{i,j,k} \) is a homomorphism of \( R\llbracket G \rrbracket \)-modules, where \( A_{i,j} \otimes_R A_{j,k} \) carries the diagonal \( G \)-action.
\end{itemize}
\end{definition}

\par We set $\mathcal{U}=\mathcal{U}(\mathcal{A})$ to consist of matrices $a=(a_{i,j})\in \mathcal{A}$ with $a_{i,i}=1$ for $i=1, \dots, n$. Let $\mathcal{U}^{\prime}=\mathcal{U}^{\prime}(\mathcal{A})$ denote the quotient of $\mathcal{U}$ by its central subgroup $\mathcal{Z}=\mathcal{Z}(\mathcal{A})$ of unipotent central elements, that is, those $a \in \mathcal{U}$ with $a_{i, j}=0$ for all $(i, j) \neq(1, n)$. 

\par Let $T_{1}, \ldots, T_{n}$ be compact $R \llbracket G \rrbracket$ modules that are $R$-finitely generated for simplicity, and let $\chi_{i}: G \rightarrow T_{i}$ be continuous 1-cocycles for $1 \leq i \leq n$.

\begin{definition}
    A defining system for the Massey product of $\chi_1, \ldots, \chi_n$ consists of a continuous 1-cocycle $\rho: G \to \mathcal{U}'$, where $\mathcal{U}'$ is the group of units in a profinite $(n+1)$-dimensional upper-triangular $(R, G)$-UGMA $\mathcal{A}$, such that $A_{i,i+1} = T_i$ for $1 \leq i \leq n$, and the composition of $\rho$ with the projection to $A_{i,i+1}$ equals $\chi_i$.

Given such a defining system $\rho$, there exists a unique lift $\tilde{\rho}: G \to \mathcal{U}$ to the full unit group of $\mathcal{A}$, such that the $(1, n+1)$-entry of $\tilde{\rho}(g)$ is zero for all $g \in G$. For each $i < j$, define $\rho_{i,j}: G \to A_{i,j}$ to be the map sending $g$ to the $(i,j)$-entry of $\tilde{\rho}(g)$.

\end{definition}

Given a defining system $\rho$, the $n$-fold Massey product $\left(\chi_{1}, \ldots, \chi_{n}\right)_{\rho} \in H^{2}\left(G, A_{1, n+1}\right)$ is the class of the 2-cocycle
$$
(g, h) \mapsto \sum_{i=2}^{n} \varphi_{1, i, n+1}\left(\rho_{1, i}(g) \otimes g \rho_{i, n+1}(h)\right)
$$
that sends $(g, h)$ to the $(1, n+1)$-entry of $\tilde{\rho}(g) \cdot g \tilde{\rho}(h)$. The Massey product $\left(\chi_{1}, \ldots, \chi_{n}\right)_{\rho}$ vanishes if and only if $\rho$ may be lifted to a cocycle $\widetilde{\rho}:G\rightarrow \mathcal{U}$.

\par We shall in our applications work with certain defining systems that can be associated with a finitely generated $R$-module $M$. Let $m$ be a positive integer less than $n$. We define an $n$-dimensional $R$-UGMA $\mathcal{A}_{n}(M, m)$ as follows. Set
$$
A_{i, j}=\left\{\begin{array}{lc}
M & \text { if } i \leq m<j \\
R & \text { otherwise }
\end{array}\right.
$$
and take the maps $\varphi_{i, j, k}$ to be the $R$-module structure maps. This makes sense since, given $i \leq j \leq k$, at least one of $A_{i, j}$ and $A_{j, k}$ must be $R$, as $m$ cannot satisfy both $m<j$ and $j \leq m$.

Let us write $\mathcal{U}_{n}(M, m)$ for $\mathcal{U}\left(\mathcal{A}_{n}(M, m)\right)$ and $\mathcal{U}_{n}^{\prime}(M, m)$ for $\mathcal{U}^{\prime}\left(\mathcal{A}_{n}(M, m)\right)$. To make this easier to visualize, note that we can write $\mathcal{U}_{n}(M, m)$ in 'block matrix' form as
$$
\mathcal{U}_{n}(M, m)=\left(\begin{array}{cc}
\mathrm{U}_{m}(R) & \mathrm{M}_{m, n-m}(M) \\
0 & \mathrm{U}_{n-m}(R)
\end{array}\right)
$$
where $\mathrm{U}_{k}(R) \leqslant \mathrm{GL}_{k}(R)$ denotes the group of upper-triangular unipotent matrices and $\mathrm{M}_{k, l}(M)$ denotes the additive group of $k$-by- $l$ matrices with entries in $M$ for positive integers $k$ and $l$. 

\par Let $T$ be an $R\llbracket G\rrbracket$-module that is finitely generated as an $R$-module. Let $n \geq 2$ be an integer, and let $a, b \geq 0$ be such that $a + b = n$. For any profinite $R\llbracket G\rrbracket$-module $M$, denote by $Z^i(G, M)$ the $R$-module of continuous $i$-cocycles $f: G \to M$. Let $\alpha = (\alpha_1, \dots, \alpha_a) \in Z^1(G, R)^a$ and $\beta = (\beta_1, \dots, \beta_b) \in Z^1(G, R)^b$.

\begin{definition}
An \emph{$(a,b)$-partial defining system} for $(n+1)$-fold Massey products of the form $(\alpha, \cdot, \beta)$ consists of a pair of continuous homomorphisms
\[
\phi: G \to U_{a+1}(R) \quad \text{and} \quad \theta: G \to U_{b+1}(R),
\]
such that the entries along the first superdiagonal of $\phi$ and $\theta$ are given by $\phi_{i, i+1} = \alpha_i$ and $\theta_{i, i+1} = \beta_i$, respectively.

Given a continuous $1$-cocycle $\lambda: G \to T$, a \emph{proper defining system} for the $(n+1)$-fold Massey product $(\alpha, \lambda, \beta)$ relative to the partial defining system $(\phi, \theta)$ is a continuous $1$-cocycle
\[
\tau: G \to \mathcal{U}_{n+2}'(T, a+1)
\]
of the form
\[
\tau = \begin{pmatrix}
\phi & \kappa \\
0 & \theta
\end{pmatrix},
\]
where $\kappa: G \to M_{a+1, b+1}'(T)$ is a continuous map satisfying $\kappa_{a+1,1} = \lambda$.
\end{definition}
\par Proper defining systems can be interpreted as cocycles, which we now describe. Let $\mathfrak{U}_{n+2}(R)$ consisting of strictly upper-triangular $(n+2) \times (n+2)$ matrices with entries in $R$. The group $\mathrm{U}_{n+2}(R)$ of upper triangular $(n+2)\times (n+2)$ unipotent matrices acts continuously on $\mathfrak{U}_{n+2}(R)$ via conjugation. Let $\mathfrak{U}_{a,b}(R)$ denote the $R\llbracket \mathrm{U}_{n+2}(R) \rrbracket$-submodule of $\mathfrak{U}_{n+2}(R)$ given by
    \[
    \mathfrak{U}_{a,b}(R) = \begin{pmatrix}
    0 & M_{a+1, b+1}(R) \\
    0 & 0
    \end{pmatrix}.
    \]
\noindent Equip $\mathfrak{U}_{a,b}(R)$ with a $G$-module structure via the continuous homomorphism
\[
G \to \mathrm{U}_{n+2}(R), \quad g \mapsto \begin{pmatrix}
\phi(g) & 0 \\
0 & \theta(g)
\end{pmatrix}.
\]
\noindent As an $R$-module,
    \[
    \mathfrak{U}_{\phi, \theta}(T) = \mathfrak{U}_{a,b}(R) \otimes_R T.
    \]
    \noindent with $G$ acting via
    \[
    g \star x = \phi(g) \cdot (g x) \cdot \theta(g)^{-1},
    \]
    where $x \in M_{a+1, b+1}(T)$, $g x$ denotes the $G$-action on entries of $x$, and $\cdot$ is matrix multiplication. Note that $\mathfrak{U}_{\phi, \theta}(T)$ contains a copy of $T$ as an $R \llbracket G \rrbracket$-submodule, embedded as the $(1, b+1)$. We set $\mathfrak{U}_{\phi, \theta}^{\prime}(T) := \mathfrak{U}_{\phi, \theta}(T) \big/ T$.

\begin{lemma}\label{lemma def system cocycles}
  Let $(\phi, \theta)$ be a partial defining system for the Massey products $(\alpha, \cdot, \beta)$. Then the map that sends a continuous $1$-cocycle $\kappa': G \rightarrow \mathfrak{U}_{\phi, \theta}^{\prime}(T)$ to the map $\tau: G \rightarrow \mathcal{U}_{n+2}^{\prime}(T, a)$ given by $
\tau = \begin{pmatrix}
\phi & \kappa' \theta \\
0 & \theta
\end{pmatrix}
$ is a bijection between $Z^{1}(G, \mathfrak{U}_{\phi, \theta}^{\prime}(T))$ and the set of proper defining systems in $T$ relative to $(\phi, \theta)$.

\end{lemma}

\begin{proof}
    This result is \cite[Lemma 3.3.3]{Lametal}.
\end{proof}
\noindent Let $e \in \mathfrak{U}_{a,b}(R)$ denote the matrix whose $(a+1,1)$-entry is $1$ and all other entries are zero. There exists a continuous $R \llbracket G \rrbracket$-module homomorphism
$$
p_{\phi,\theta} : \Omega / I^{n+1} \longrightarrow \mathfrak{U}_{a,b}(R)
$$
\noindent defined on the cosets of images of group elements by
$$
p_{\phi,\theta}([h]) = \phi(h) \, e \, \theta(h)^{-1}.
$$
\noindent Moreover, the image of $I^{n}$ is contained in the submodule of matrices whose entries vanish outside the $(1,b+1)$-position. The following result shows that the image of $\Psi^{(n)}$ is closely related to certain Massey products. However, due to the dependence on the map $p_{\phi, \theta}$, this relationship does not, in general, fully determine the image of $\Psi^{(n)}$.

\begin{theorem}
  Suppose $\phi: G \to \mathrm{U}_{a+1}(R)$ and $\theta: G \to \mathrm{U}_{b+1}(R)$ restrict to the cocycles $\alpha \in Z^1(G, R)^a$ and $\beta \in Z^1(G, R)^b$, respectively, as described above. Let $f \in Z^1\left(G, T \otimes_{\mathbb{F}_p} \Omega / I^n\right)$, and let $\rho$ be the proper defining system relative to $(\phi, \theta)$ associated with $p_{\phi, \theta} \circ f$, as in Lemma \ref{lemma def system cocycles}. Then we have the identity
$$
p_{\phi, \theta}\left(\Psi^{(n)}([f])\right) = \left( \alpha, \left(p_{\phi, \theta} \circ f\right)_{a+1,1}, \beta \right)_\rho
$$
\noindent in $H^2(G, T)$. 
\end{theorem}
\begin{proof}
    See \cite[Theorem 4.1.2]{Lametal}.
\end{proof}
\par We now assume that $G$ has $p$ cohomological dimension $2$ and that $H\simeq \Z_p$. Let $\kappa: G\twoheadrightarrow \Z_p$ be a surjective map which factors through the quotient map $G\twoheadrightarrow H$. In the special case where $\alpha = (\kappa, \kappa, \dots, \kappa)$, partial defining systems can be constructed explicitly using \emph{unipotent binomial matrices}, which we now recall. Let $u_{n}$ denote the $(n+1)$-dimensional nilpotent upper triangular matrix
$$
u_{n}=\left(\begin{array}{ccccc}
0 & 1 & 0 & \cdots & 0 \\
& 0 & 1 & \ddots & \vdots \\
& & 0 & \ddots & 0 \\
& & & \ddots & 1 \\
& & & & 0
\end{array}\right) .
$$
\noindent Let
$
\left[\begin{array}{c}\cdot \\ n\end{array}\right] : \mathbb{Z}_p \to \mathrm{U}_{n+1}(\mathbb{Z}_p)
$ denote the unique continuous group homomorphism such that $
\left[\begin{array}{c} 1 \\ n \end{array}\right] = 1 + u_n
$. Then for any $m \in \mathbb{Z}$, we have
$$
\left[\begin{array}{c} m \\ n \end{array}\right] = (1 + u_n)^m = \sum_{k=0}^{n} \binom{m}{k} u_n^k,
$$
\noindent which equals
$$
\begin{pmatrix}
1 & m & \binom{m}{2} & \cdots & \binom{m}{n} \\
& 1 & m & \ddots & \vdots \\
& & 1 & \ddots & \binom{m}{2} \\
& & & \ddots & m \\
& & & & 1
\end{pmatrix}.
$$
\noindent Let $(\phi, \theta)$ be the $(n, 0)$-proper defining system with $\phi=\left[\begin{array}{c}\kappa \\ n\end{array}\right].$ We have $\alpha=(\kappa, \ldots, \kappa) \in Z^{1}(G, R)^{n}$, which we denote by $\kappa^{(n)}$. We denote $\mathfrak{U}_{\phi, \theta}(T)$ by $\mathfrak{U}_{\left[\begin{array}{l}\kappa \\ n\end{array}\right]}(T)$. the proper defining system $\tau_{x^{n}}$ relative to $\left[\begin{array}{c}\kappa \\ n\end{array}\right]$ that is attached to $p_{n} \circ f$, where
$$
f=\sum_{k=0}^{n-1} \lambda_{k} x^{k} \in Z^{1}\left(G, T \otimes_{\F_p} \Omega / I^{n}\right),
$$
satisfies $\left(\tau_{x^{n}}\right)_{n+1-k, n+2}=\lambda_{k}$ for $0 \leq k \leq n-1$. In particular, the element $\lambda=\lambda_{0}=\left(\tau_{x^{n}}\right)_{n+1, n+2}$ is the image of $f$ under the map
$$
Z^{1}\left(G, T \otimes_{\F_p} \Omega / I^{n}\right) \rightarrow Z^{1}(G, T)
$$
induced by the augmentation $\Omega / I^{n} \mapsto \Omega / I=R$.

\begin{theorem}
    For $f \in Z^{1}\left(G, T \otimes_{\F_p} \Omega / I^{n}\right)$, we have
$$
\Psi^{(n)}([f])=\left(\kappa^{(n)}, \lambda\right)_{\tau_{x^{n}}} \cdot x^{n}
$$
where $\tau_{x^{n}}$ is the proper defining system relative to $\left[\begin{array}{c}\kappa \\ n\end{array}\right]$ attached to $f$, and $\lambda$ is the image of $f$ in $Z^{1}(G, T)$.
\end{theorem}

\begin{proof}
    This result is \cite[Theorem 4.3.1]{Lametal}.
\end{proof}

\begin{corollary}\label{filtration corollary}
    Suppose that $G$ is p-cohomologically finite of $p$-cohomological dimension $2$ and that $H\simeq \Z_p$. Let $P_{n}(H)$ denote the subgroup of $H^{2}(G, T) \otimes_{\F_p} I^{n} / I^{n+1}$ generated by all $\left(\kappa^{(n)}, \lambda\right)_{\tau} \cdot x^{n}$ for proper defining systems $\tau$ relative to $\left[\begin{array}{c}\kappa \\ n\end{array}\right]$ and $\lambda=\tau_{n+1, n+2}$. We have a canonical isomorphism of $R$-modules
$$
\frac{I^{n} H_{\mathrm{Iw}}^{2}(N, T)}{I^{n+1} H_{\mathrm{Iw}}^{2}(N, T)} \cong \frac{H^{2}(G, T) \otimes_{\F_p} I^{n} / I^{n+1}}{P_{n}(H)}.
$$
\end{corollary}

\begin{proof}
    See \cite[Corollary 4.3.3]{Lametal}.
\end{proof}

\section{Main results}
\par In this section, we establish the main results of this paper. The following Theorem establishes a criterion for the vanishing of the $\mu$-invariant of the fine Selmer group stated in purely in terms of Massey products. This result has striking consequences to establishing the constancy of the $\mu$-invariant in Greenberg neighbourhoods of a $\Z_p$-extensions of a given one, which we will subsequently explore.
\begin{theorem}\label{main thm of paper}
Let $K$ be a number field and $K_\infty/K$ be a $\Z_p$-extension and $S$ be a finite set of primes containing the primes of $K$ which lie above $p$. Assume that all primes $v|p$ of $K$ are totally ramified in $K_\infty$ and all primes in $S$ are finitely decomposed in $K_\infty$. Let $\cK$ be a finite extension of $\Q_p$ with valuation ring $\cO$ and let $\rho:\op{G}_{K,S}\rightarrow \op{GL}_d(\cO)$ be a continuous Galois representation. Then the following are equivalent. 
\begin{enumerate}
    \item\label{main thm of paper c1} The fine Selmer group $\mathcal{R}_{p^\infty}(\mathbf{A}(\rho)/K_\infty)$ is cotorsion with vanishing $\mu$-invariant. 
    \item\label{main thm of paper c2} There exists $n>0$ such that the Massey products $(\kappa^{(n)}, \lambda)_\tau$ span $H^2(K_S/K, \mathbf{V}_{\bar{\rho}^*})$, where $\tau$ ranges over defining systems relative to $\left[\begin{array}{c}\kappa \\ n\end{array}\right]$ and $\lambda=\tau_{n+1, n+2}\in H^1(K_S/K, \mathbf{V}_{\bar{\rho}^*})$. 
\end{enumerate}
\end{theorem}
\begin{proof}
    The $\mu$-invariant of the fine Selmer group $\mathcal{R}_{p^\infty}(\mathbf{A}(\rho)/K_\infty)$ vanishes if and only if the Iwasawa cohomology group $H^2_{\op{Iw}}(K_S/K_{\infty}, \mathbf{V}_{\bar{\rho}^*})$ is finite. Thus, (1) is equivalent to the condition that there exists some $n>0$ such that \[I^n H^2_{\op{Iw}}(K_S/K_{\infty}, \mathbf{V}_{\bar{\rho}^*})=I^{n+1} H^2_{\op{Iw}}(K_S/K_{\infty}, \mathbf{V}_{\bar{\rho}^*}).\] It then follows from Corollary \ref{filtration corollary} that (1) and (2) are equivalent.
\end{proof}

\begin{remark}\label{remark qi}
    At this point, some remarks are in order. 
    \begin{itemize}
        \item Specialize to the case when $\rho=1$ is the trivial $1$-dimensional representation, $\mathcal{R}_{p^\infty}(\mathbf{A}(\rho)/K_\infty)$ is the Pontryagin dual of $\mathcal{Y}_S$, considered in Remark \ref{remark on rho=1}. Theorem \ref{main thm of paper} then extends the $\mu$-invariant vanishing criterion in \cite{qi2025iwasawa}, where it is assumed that $H^2(K_S/K, \mu_p)$ is $1$-dimensional. 
        \item When one specializes to the case when $p$ is a regular prime, $K=\Q(\mu_p)$ and $K_\infty=\Q(\mu_{p^\infty})$ is the cyclotomic $\Z_p$-extension of $K$, the maximal pro-$p$ quotient of $\op{Gal}(K_S/K)$ is free pro-$p$. In general, this quotient is minimal generated by $r:=\dim_{\F_p} H^1(K_S/K, \F_p)$ elements modulo $s:=\dim_{\F_p} H^2(K_S/K, \F_p)$ relations. When $p$ is a regular prime, 
        \[H^2(K_S/K, \mathbf{V}_{\bar{\rho}^*})=H^2(K_S/K, \mu_p)=H^2(K_S/K, \F_p)=0\] and the condition (2) in Theorem \ref{main thm of paper} is vacuously true. 
        \item Suppose that $\rho_1$ and $\rho_2$ are two representations $\op{G}_{K,S}\rightarrow \op{GL}_d(\cO)$ such that $\bar{\rho}_1\simeq \bar{\rho}_2$. Then, condition (1) is true for $\rho_1$ if and only if condition (1) is true for $\rho_2$. Indeed, condition (2) depends only on the residual representation. This is seen more directly from Theorem \ref{Conj A criterion}.
        \end{itemize}
\end{remark} 

\par Let $K$ be a number field of degree $n=[K:\Q]=r_1+2r_2$, with $r_1$ real places and $r_2$ pairs of complex conjugate places. Fix a prime $p$. Let $\widetilde{K}$ denote the compositum of all $\Z_p$-extensions of $K$. Then $\operatorname{Gal}(\widetilde{K}/K)\simeq \Z_p^{t}$, where $ 
t=r_2+1+\delta_p(K)$, and $\delta_p(K)$ is the Leopoldt defect of $K$ at $p$. The standard bound $\delta_p(K)\le r_1+r_2-1$ yields $
r_2+1\le t\le r_1+2r_2=n
$. Write $\widetilde{\mathcal{F}}$ for the maximal abelian unramified pro-$p$ extension of $\widetilde{K}$, and set $\widetilde{\mathcal{X}}:=\operatorname{Gal}(\widetilde{\mathcal{F}}/\widetilde{K})$. The Iwasawa algebra attached to $\Gamma$ is
\[\widetilde{\Lambda}:=\Z_p\llbracket \op{Gal}(\widetilde{K}/K)\rrbracket.\]
\noindent After choosing topological generators $\sigma_1,\dots,\sigma_t$ of $\Gamma$ and writing $T_i:=\sigma_i-1$, we obtain a noncanonical identification
$
\widetilde{\Lambda}\simeq \mathbb{Z}_p\llbracket T_1,\dots,T_t\rrbracket
$. Via the natural conjugation action of $\op{Gal}(\widetilde{K}/K)$ on $\widetilde{\mathcal{F}}$, the group $\widetilde{\mathcal{X}}$ acquires a canonical structure of a $\widetilde{\Lambda}$-module. A theorem of Greenberg \cite[Theorem 1]{GreenbergAJM} shows that $\widetilde{\mathcal{X}}$ is in fact a Noetherian torsion $\widetilde{\Lambda}$-module. 
\par Let $\mathscr{E}$ be the set of all $\Z_p$-extensions of $K$. Let $K_\infty^1$ and $K_\infty^2$ be $\Z_p$-extensions of $K$. Let $K_n^i/K$ be the $n$-th layer of $K_\infty^i$. Set
\[n_0(K_\infty^1, K_\infty^2):=\op{max}\{ n\mid K_n^1=K_n^2\}\] with the convention that $n_0(K_n^1, K_n^2):=\infty$ if $K_\infty^1=K_\infty^2$. The distance between $K_\infty^1$ and $K_\infty^2$ is then defined as follows:
\[\delta(K_\infty^1, K_\infty^2):=\begin{cases}
    p^{-n_0(K_\infty^1, K_\infty^2)} & \text{ if }K_\infty^1\neq  K_\infty^2,\\
    0 & \text{ if }K_\infty^1=  K_\infty^2.
\end{cases}\]
This gives a topology on $\mathscr{E}$, it is easy to see that $\mathscr{E}$ is indeed compact and hausdorff. Let $K_\infty\in \mathscr{E}$ and let $n$ be a positive integer. Let $\mathscr{E}(K_\infty, n)$ consist of all $K_\infty'\in \mathscr{E}$ such that 
\[[K_\infty'\cap K_\infty: K]\geq p^n,\] or equivalently, 
\[\delta(K_\infty', K_\infty)\leq p^{-n}.\] Then $\{\mathscr{E}(K_\infty, n)\mid K_\infty\in \mathscr{E}, n\in \Z_{\geq 1}\}$ gives a basis for the topology on $\mathscr{E}$. We refer to a basic open set of the form $\mathscr{E}(K_\infty, n)$ is referred to as a \emph{Greenberg neighbourhood} of $K_\infty$. 
\par Let $Y$ be a Noetherian torsion $\widetilde{\Lambda}$-module. Let $K_\infty\in \mathscr{E}$ and let $\Lambda_{K_\infty}$ be the associated Iwasawa algebra. Denote by $\widetilde{\Lambda}\twoheadrightarrow \Lambda_{K_\infty}$ the natural surjection and set $Y_{K_\infty}:=Y\otimes_{\widetilde{\Lambda}} \Lambda_{K_\infty}$. Then $Y_{K_\infty}$ is Noetherian as a $\Lambda_{K_\infty}$-module but need not be torsion. Let $\mathscr{E}(Y)$ be the set of $\Z_p$-extensions $K_\infty$ for which $Y_{K_\infty}$ is a torsion $\Lambda_{K_\infty}$-module. Then, $\mathscr{E}(Y)$ is an open subset of $\mathscr{E}$ (cf. \cite[Proposition 1]{GreenbergAJM}). For $K_\infty\in \mathscr{E}(Y)$, denote by $\mu(Y_{K_\infty})$ the $\mu$-invariant of $Y_{K_\infty}$. This defines a function
\[K_\infty\rightarrow \mu(Y_{K_\infty})\] on $\mathscr{E}(Y)$, which Greenberg proved is locally bounded (see \cite[Theorem 2]{GreenbergAJM}). Moreover if the $\mu$-invariant vanishes at $K_\infty$, then it vanishes on a Greenberg neighbourhood of $K_\infty$ (cf. \cite[Theorem 3]{GreenbergAJM}). The following result shows that if two $\Z_p$-extensions $K_\infty^1$ and $K_\infty^2$ agree at the first layer itself, this can be sufficient condition to deduce that
\[\mu\left(\mathcal{R}_{p^\infty}(\mathbf{A}(\rho)/K_\infty^1)\right)=0\Leftrightarrow \mu\left(\mathcal{R}_{p^\infty}(\mathbf{A}(\rho)/K_\infty^2)\right)=0.\]

\begin{theorem}\label{greenberg neighbourhood theorem}
    Let $K$ be a number field and let $S$ be a finite set of primes of $K$ containing the primes of $K$ which lie above $p$. Let $\rho: \op{G}_{K, S}\rightarrow \op{GL}_d(\cO)$ be a continuous Galois representation. Let $K_\infty^1$ and $K_\infty^2$ be two $\Z_p$-extensions of $K$ such that:
    \begin{enumerate}
        \item the primes $v|p$ of $K$ are totally ramified in both $K_\infty^1$ and $K_\infty^2$, 
        \item all primes $v\in S$ are finitely decomposed in $K_\infty^1$ and $K_\infty^2$, 
        \item the first layers of $K_\infty^1$ and $K_\infty^2$ coincide, i.e., $K_1^1=K_1^2$.
    \end{enumerate}
    Suppose that $\mathcal{R}_{p^\infty}(\mathbf{A}(\rho)/K_\infty^1)$ is cotorsion with $\mu\left(\mathcal{R}_{p^\infty}(\mathbf{A}(\rho)/K_\infty^1)\right)=0$. Then, $\mathcal{R}_{p^\infty}(\mathbf{A}(\rho)/K_\infty^2)$ is cotorsion with $\mu\left(\mathcal{R}_{p^\infty}(\mathbf{A}(\rho)/K_\infty^2)\right)=0$.
\end{theorem}

\begin{proof}
    Theorem \ref{main thm of paper} applies to both $(\rho, K_\infty^1)$ and $(\rho, K_\infty^2)$. Condition \eqref{main thm of paper c2} is satisfied for $(\rho, K_\infty^1)$ if and only if it is satisfied for $(\rho, K_\infty^2)$. This is because the Massey products are taken with respect to $\mathbb{F}_p$-coefficients and thus only depend on the reduction of $\kappa$ modulo $p$. Since the first layers of the two extensions coincide, the Massey products in the formulation of \eqref{main thm of paper c2} also coincide. Thus by Theorem \ref{main thm of paper}, we find that condition \eqref{main thm of paper c1} of Theorem \ref{main thm of paper} is true for $(\rho, K_\infty^1)$ if and only if it is true for $(\rho, K_\infty^2)$.
\end{proof}
\noindent The result is significant because it provides an explicit Greenberg neighbourhood of a $\Z_p$-extension in which the $\mu=0$ condition continues to hold. This goes beyond Greenberg’s general theorem, which guarantees only the existence of some neighbourhood where the $\mu=0$ condition is satisfied, without specifying it explicitly.

\bibliographystyle{alpha}
\bibliography{references}

@article {Lametal,
    AUTHOR = {Lam, Yeuk Hay Joshua and Liu, Yuan and Sharifi, Romyar and
              Wake, Preston and Wang, Jiuya},
     TITLE = {Generalized {B}ockstein maps and {M}assey products},
   JOURNAL = {Forum Math. Sigma},
  FJOURNAL = {Forum of Mathematics. Sigma},
    VOLUME = {11},
      YEAR = {2023},
     PAGES = {Paper No. e5, 41},
   MRCLASS = {20J05 (11R23 11R34 12G05 20J06)},
  MRNUMBER = {4537772},
MRREVIEWER = {Meng Fai Lim},
       DOI = {10.1017/fms.2022.103},
       URL = {https://doi.org/10.1017/fms.2022.103},
}

@book {perrinriou,
    AUTHOR = {Perrin-Riou, B.},
     TITLE = {{$p$}-adic {$L$}-functions and {$p$}-adic representations},
    SERIES = {SMF/AMS Texts and Monographs},
    VOLUME = {3},
      NOTE = {Translated from the 1995 French original by Leila Schneps and
              revised by the author},
 PUBLISHER = {American Mathematical Society, Providence, RI; Soci\'{e}t\'{e}
              Math\'{e}matique de France, Paris},
      YEAR = {2000},
     PAGES = {xx+150},
      ISBN = {0-8218-1946-1},
   MRCLASS = {11G40 (11F33 11G45 11R23 11R42 11S25)},
  MRNUMBER = {1743508},
}

@incollection {Massey,
    AUTHOR = {Massey, W. S.},
     TITLE = {Some higher order cohomology operations},
 BOOKTITLE = {Symposium internacional de topolog\'{\i}a algebraica
              {I}nternational symposium on algebraic topology},
     PAGES = {145--154},
 PUBLISHER = {Universidad Nacional Aut\'{o}noma de M\'{e}xico and UNESCO, M\'{e}xico},
      YEAR = {1958},
   MRCLASS = {55.00},
  MRNUMBER = {98366},
MRREVIEWER = {Edgar H. Brown, Jr.},
}

@article {ochivenjakob,
    AUTHOR = {Ochi, Yoshihiro and Venjakob, Otmar},
     TITLE = {On the structure of {S}elmer groups over {$p$}-adic {L}ie
              extensions},
   JOURNAL = {J. Algebraic Geom.},
  FJOURNAL = {Journal of Algebraic Geometry},
    VOLUME = {11},
      YEAR = {2002},
    NUMBER = {3},
     PAGES = {547--580},
}

@article {CoatesSujatha,
    AUTHOR = {Coates, J. and Sujatha, R.},
     TITLE = {Fine {S}elmer groups of elliptic curves over {$p$}-adic {L}ie
              extensions},
   JOURNAL = {Math. Ann.},
  FJOURNAL = {Mathematische Annalen},
    VOLUME = {331},
      YEAR = {2005},
    NUMBER = {4},
     PAGES = {809--839},
}

@article {BCM1,
    AUTHOR = {Burungale, Ashay and Clozel, Laurent},
     TITLE = {Ordinary deformations are unobstructed in the cyclotomic
              limit},
   JOURNAL = {Asian J. Math.},
  FJOURNAL = {Asian Journal of Mathematics},
    VOLUME = {27},
      YEAR = {2023},
    NUMBER = {3},
     PAGES = {405--422},
      ISSN = {1093-6106},
   MRCLASS = {11F80 (11R23 11S25)},
  MRNUMBER = {4671900},
MRREVIEWER = {Ramdorai Sujatha},
       DOI = {10.4310/ajm.2023.v27.n3.a4},
       URL = {https://doi.org/10.4310/ajm.2023.v27.n3.a4},
}

@phdthesis{qi2025iwasawa,
  title={Iwasawa $\lambda$ Invariants, Massey Products, and Pseudo-Null Modules},
  author={Qi, Peikai},
  year={2025},
  school={Michigan State University}
}

@article {SharifiMassey,
    AUTHOR = {Sharifi, Romyar T.},
     TITLE = {Massey products and ideal class groups},
   JOURNAL = {J. Reine Angew. Math.},
  FJOURNAL = {Journal f\"{u}r die Reine und Angewandte Mathematik. [Crelle's
              Journal]},
    VOLUME = {603},
      YEAR = {2007},
     PAGES = {1--33},
}

@article {MinacTan1,
    AUTHOR = {Min\'{a}\v{c}, J\'{a}n and T\^{a}n, Nguyen Duy},
     TITLE = {Triple {M}assey products over global fields},
   JOURNAL = {Doc. Math.},
  FJOURNAL = {Documenta Mathematica},
    VOLUME = {20},
      YEAR = {2015},
     PAGES = {1467--1480},
      ISSN = {1431-0635},
}

@article {MinacTan2,
    AUTHOR = {Min\'{a}\v{c}, J\'{a}n and T\^{a}n, Nguyen Duy},
     TITLE = {Triple {M}assey products over global fields},
   JOURNAL = {Doc. Math.},
 YEAR = {2015},
}

@article {MinacTan3,
    AUTHOR = {Min\'{a}\v{c}, J\'{a}n and T\^{a}n, Nguyen Duy},
     TITLE = {Triple {M}assey products and {G}alois theory},
   JOURNAL = {J. Eur. Math. Soc. (JEMS)},
  FJOURNAL = {Journal of the European Mathematical Society (JEMS)},
    VOLUME = {19},
      YEAR = {2017},
    NUMBER = {1},
     PAGES = {255--284},
      ISSN = {1435-9855},
}

@article {MinacTan4,
    AUTHOR = {Min\'{a}\v{c}, J\'{a}n and T\^{a}n, Nguyen Duy},
     TITLE = {Counting {G}alois {$\Bbb{U}_4(\Bbb{F}_p)$}-extensions using
              {M}assey products},
   JOURNAL = {J. Number Theory},
  FJOURNAL = {Journal of Number Theory},
    VOLUME = {176},
      YEAR = {2017},
     PAGES = {76--112},
      ISSN = {0022-314X},
}

@article{BCM2,
  title={Dimension of the deformation space of ordinary representations in the cyclotomic limit},
  author={Burungale, Ashay A and Clozel, Laurent and Mazur, Barry},
  journal={arXiv preprint arXiv:2406.02473},
  year={2024}
}

@book {Hida,
    AUTHOR = {Hida, Haruzo},
     TITLE = {Hilbert modular forms and {I}wasawa theory},
    SERIES = {Oxford Mathematical Monographs},
 PUBLISHER = {The Clarendon Press, Oxford University Press, Oxford},
      YEAR = {2006},
}

@article {WWE,
    AUTHOR = {Wake, Preston and Wang-Erickson, Carl},
     TITLE = {The rank of {M}azur's {E}isenstein ideal},
   JOURNAL = {Duke Math. J.},
  FJOURNAL = {Duke Mathematical Journal},
    VOLUME = {169},
      YEAR = {2020},
    NUMBER = {1},
     PAGES = {31--115},
}

@article {GreenbergAJM,
    AUTHOR = {Greenberg, Ralph},
     TITLE = {The {I}wasawa invariants of {${\bf \Gamma }$}-extensions of a
              fixed number field},
   JOURNAL = {Amer. J. Math.},
  FJOURNAL = {American Journal of Mathematics},
    VOLUME = {95},
      YEAR = {1973},
     PAGES = {204--214},
      ISSN = {0002-9327},
}

@article {sharifieisenstein,
    AUTHOR = {Sharifi, Romyar T.},
     TITLE = {Iwasawa theory and the {E}isenstein ideal},
   JOURNAL = {Duke Math. J.},
  FJOURNAL = {Duke Mathematical Journal},
    VOLUME = {137},
      YEAR = {2007},
    NUMBER = {1},
     PAGES = {63--101},
}

@article {PAllen,
    AUTHOR = {Allen, Patrick B.},
     TITLE = {Deformations of polarized automorphic {G}alois representations
              and adjoint {S}elmer groups},
   JOURNAL = {Duke Math. J.},
  FJOURNAL = {Duke Mathematical Journal},
    VOLUME = {165},
      YEAR = {2016},
    NUMBER = {13},
     PAGES = {2407--2460},
}

@book{introcycfields,
    AUTHOR = {Washington, Lawrence C.},
     TITLE = {Introduction to cyclotomic fields},
    SERIES = {Graduate Texts in Mathematics},
    VOLUME = {83},
 PUBLISHER = {Springer-Verlag, New York},
      YEAR = {1982},
     PAGES = {xi+389},
}

@book {NSW,
    AUTHOR = {Neukirch, J\"{u}rgen and Schmidt, Alexander and Wingberg, Kay},
     TITLE = {Cohomology of number fields},
    SERIES = {Grundlehren der mathematischen Wissenschaften [Fundamental
              Principles of Mathematical Sciences]},
    VOLUME = {323},
 PUBLISHER = {Springer-Verlag, Berlin},
      YEAR = {2000},
}

@article {McCallumSharifi,
    AUTHOR = {McCallum, William G. and Sharifi, Romyar T.},
     TITLE = {A cup product in the {G}alois cohomology of number fields},
   JOURNAL = {Duke Math. J.},
  FJOURNAL = {Duke Mathematical Journal},
    VOLUME = {120},
      YEAR = {2003},
    NUMBER = {2},
     PAGES = {269--310},
}

@article {IwasawaAnnals,
    AUTHOR = {Iwasawa, Kenkichi},
     TITLE = {On {${\bf Z}_{l}$}-extensions of algebraic number fields},
   JOURNAL = {Ann. of Math. (2)},
  FJOURNAL = {Annals of Mathematics. Second Series},
    VOLUME = {98},
      YEAR = {1973},
     PAGES = {246--326},
}

@article {ferrerowashington,
    AUTHOR = {Ferrero, Bruce and Washington, Lawrence C.},
     TITLE = {The {I}wasawa invariant {$\mu _{p}$} vanishes for abelian
              number fields},
   JOURNAL = {Ann. of Math. (2)},
  FJOURNAL = {Annals of Mathematics. Second Series},
    VOLUME = {109},
      YEAR = {1979},
    NUMBER = {2},
     PAGES = {377--395},
}

@article {DRS,
    AUTHOR = {Deo, Shaunak V. and Ray, Anwesh and Sujatha, R.},
     TITLE = {On the {$\mu$} equals zero conjecture for fine {S}elmer groups
              in {I}wasawa theory},
   JOURNAL = {Pure Appl. Math. Q.},
  FJOURNAL = {Pure and Applied Mathematics Quarterly},
    VOLUME = {19},
      YEAR = {2023},
    NUMBER = {2},
     PAGES = {641--680},
}
\end{document}